\documentclass[11pt]{amsart}
\usepackage{graphicx}
\usepackage{pgf,tikz,pgfplots}
\usepackage{mathrsfs}
\usepackage{graphicx}
\usepackage[all,cmtip]{xy}
\usepackage{amsmath}
\usepackage{tikz}
\usepackage{mathdots}
\usepackage{yhmath}
\usepackage{cancel}
\usepackage{color}
\usepackage{siunitx}
\usepackage{array}
\usepackage{multirow}
\usepackage{amssymb}
\usepackage{gensymb}
\usepackage{tabularx}
\usepackage{booktabs}
\usepackage{makecell}
\usepackage{young}
\usepackage{youngtab}
\usepackage{titletoc}
\usepackage{extarrows}
\usepackage{varioref}
\usepackage{bm}
\usepackage{ytableau}
\usepackage{array}
\usepackage{multirow}
\usetikzlibrary{fadings}
\usetikzlibrary{patterns}
\usetikzlibrary{shadows.blur}
\usetikzlibrary{shapes}
\usepackage{multicol}
\usepackage[normalem]{ulem}
\usepackage{hyperref,xcolor}
\definecolor{wine-stain}{rgb}{0.5,0,0}
\hypersetup{
	colorlinks,
	linkcolor=wine-stain,
	linktoc=all
}

\usepackage{latexsym,bm,amsmath,amssymb}

\usetikzlibrary{arrows}
\newtheorem{theorem}{Theorem}[section]
\newtheorem{corollary}{Corollary}
\newtheorem{lemma}{Lemma}
\newtheorem{proposition}{Proposition}
\theoremstyle{definition}
\newtheorem{example}{Example}
\newtheorem{definition}{Definition}

\newtheorem{assumption}{Assumption}
\newtheorem{def-thm}{Definition-Theorem}
\newtheorem{propdef}{Proposition-Definition}
\newtheorem{upshot}{Upshot}

\newtheorem{remark}{Remark}
\numberwithin{equation}{section}
\numberwithin{definition}{section}
\numberwithin{corollary}{section}
\numberwithin{proposition}{section}
\numberwithin{lemma}{section}
\numberwithin{remark}{section}
\numberwithin{propdef}{section}

\newcommand{\sA}{{\mathcal A}}
\newcommand{\sB}{{\mathcal B}}
\newcommand{\sC}{{\mathcal C}}

\newcommand{\sE}{{\mathcal E}}
\newcommand{\sF}{{\mathcal F}}
\newcommand{\sG}{{\mathcal G}}
\newcommand{\sH}{{\mathcal H}}
\newcommand{\sI}{{\mathcal I}}
\newcommand{\sJ}{{\mathcal J}}
\newcommand{\sK}{{\mathcal K}}
\newcommand{\sL}{{\mathcal L}}
\newcommand{\sM}{{\mathcal M}}

\newcommand{\sO}{{\mathcal O}}
\newcommand{\sP}{{\mathcal P}}

\newcommand{\sR}{{\mathcal R}}

\newcommand{\g}{{\mathfrak g}}

\newcommand{\gp}{\mathfrak{p}}
\newcommand{\gb}{\mathfrak{b}}

\newcommand{\gn}{\mathfrak{n}}
\newcommand{\gc}{\mathfrak{c}}
\newcommand{\gt}{\mathfrak{t}}

\newcommand{\ssC}{{\mathscr C}}

\newcommand{\BC}{\mathbb C}
\newcommand{\BD}{\mathbb D}

\newcommand{\BR}{\mathbb R}

\newcommand{\BZ}{\mathbb Z}

\newcommand{\Gal}{\operatorname{Gal}}

\newcommand{\Hig}{\operatorname{Hig}}
\newcommand{\Pic}{\operatorname{Pic}}

\newcommand{\Lhig}{{}^{\sL}\!\Hig}

\newcommand{\Lie}{\operatorname{Lie}}

\newcommand{\Ad}{\operatorname{Ad}}

\newcommand{\GL}{\operatorname{GL}}

\newcommand{\SL}{\operatorname{SL}}

\newcommand{\SO}{\operatorname{SO}}

\newcommand{\Al}{\mathit{Al}}

\newcommand{\Hom}{\operatorname{Hom}}
\newcommand{\Mor}{\operatorname{Mor}}

\newcommand{\ssm}{\smallsetminus}

\newcommand{\Spec}{\operatorname{Spec}}

\newcommand{\Aut}{\operatorname{Aut}}

\newcommand{\fppf}{{\operatorname{fppf}}}

\newcommand{\Res}{\operatorname{Res}}

\newcommand{\chr}{\operatorname{char}}

\newcommand{\Bun}{\operatorname{Bun}}

\begin{document}

\title[]{Generic Fibers of Parahoric Hitchin Systems}%
\author{Bin Wang}%
\address{{Address of the Author: Yau Mathematical Science Center, Tsinghua University, Beijing, 100084, China.}}%
\email{\href{mailto:email address}{{wangbin15@mails.tsinghua.edu.cn}}}%
\thanks{}%
\subjclass{}%
\keywords{}%

\begin{abstract}
	In this paper, we talk about parahoric Hitchin systems over smooth projective curves with structure group a semisimple simply connected group. We describe the geometry of generic fibers of parahoric Hitchin fibrations using root stacks. We work over an algebraically closed field with a mild assumption of the characteristic. All of these can be treated as a generalization of $\GL_{n}$ case in \cite{SWW19}. 
\end{abstract}
\maketitle
\section{Introduction}
\label{chapter: introduction}
We fix a semisimple simply connected group $G$ over an algebraically closed field $k$. We fix a maximal torus $T\subset G$ and a Borel subgroup $B$ containing $T$. We write $X^{*}(T)$ (resp. $X_{*}({T}))$ for the character group (resp. cocharacter group) of $T$, as usual. We denote the Weyl group $N_{G}(T)/T$ by $W$. We use the Fraktur font $\g,\gt,\gb, \dots$ to denote the Lie algebra of $G,T,B, \dots$. The Borel group determines a fundamental closed chamber $\sC$ in  $X_{*}(T)\otimes\mathbb{R}$; we denote by $\Al\subset X_{*}(T)\otimes\mathbb{R}$ the closed affine alcove which contains the origin and is contained in that chamber. Throughout the paper, we make the following assumption:
\begin{assumption}\label{blas}
	Let  $h(G)$ be the maximum of Coxeter numbers of simple factors of $\g$. If  $\chr (k)=p>0$, we assume $p>2h(G)$.
\end{assumption}
This ensures that $k[\gt]^{W}$ is a polynomial algebra that can be identified with  $k[\g]^{G}$. The underlying affine space is called the \emph{adjoint quotient} or the  \emph{Chevalley space} of $\g$ and will be denoted by $\gc$. Under this assumption, the \emph{Jacobson-Morozov lemma} (that every nonzero nilpotent element of $\g$ is part of an $\mathfrak{sl}(2)$-triple) still holds.


Let $C$ be a smooth projective curve over an algebraically closed field $k$,  for which is given a finite subset $D\subset C$ and for  each $x\in D$,  a facet $\sF(x)$ of the alcove $\Al$.  The  Bruhat-Tits theory associates to these data a smooth affine group scheme $\sG_{\sF}$ over $C$, to which we shall refer as  a \emph{Bruhat-Tits group scheme} with \emph{parahoric structure $(D,\sF)$}. We assume that $2g(C)-2+\deg D>0$. We here study the geometry of parahoric Hitchin systems over $C$ with the structure group $\sG_{\sF}$.  We give an algebro-geometric interpretation of the completely algebraically integrable system, which has been studied by Baraglia,  Kamgarpour and Varma in \cite{BKV18} over $\BC$. To be more precise,among other things, they studied parahoric Hitchin systems over a Riemann surface when $G$ is semisimple and simply connected, and proved that the generic fibers are Abelian varieties. For each facet $\sF(x)$, we may choose a rational cocharacter $\lambda(x)$ in the interior of $\sF(x)$. We take the point of view of  Balaji and Seshadri \cite{BS14} who interpreted parahoric structures as essentially orbifold structures. Our strategy is to translate parahoric Higgs bundles into Higgs bundles over a root stack $\mathscr{C}$ (which is indeed roughly speaking an analogue of an orbifold curve) associated with the parahoric data $(D,\sF)$ and a choice of $\lambda(x)\in\sF(x)$ for each $x\in D$. Now we can state our main theorem:

\begin{theorem}\label{thm:conclusion in the intro}
	Let $\mathscr{C}$ be the root stack associated with parahoric data $(D,\sF)$ and $\lambda(x)$ for $x\in D$. Then we have an equivalence between the following two categories:
	\begin{equation}
	\{\text{$G$-Higgs bundles over $\mathscr{C}$}\}\longleftrightarrow
	\{\text{ $\sG_{\sF}$-Higgs bundles over C}\}
	\end{equation}
	that is compatible with the Hitchin maps.
\end{theorem}
We might also state this theorem as identifying two moduli stacks, namely the one of Higgs bundles over $\mathscr{C}$ (denoted by $\Hig_{\mathscr{C}}$---we omit $G$ in the notation for simplicity) and the one of parahoric Higgs bundles over $C$ (denoted by $\Hig_{\sG_{\sF}}$ which then is represented by an algebraic stack. In particular, we also have the following interesting corollary:
\begin{corollary}
	The image of parahoric Hitchin map is indeed an affine space.
\end{corollary}
For a parahoric Hitchin system over a Riemann surface, with structure group a semisimple simply connected algebraic group, Baraglia, Kamgarpour and Varma in \cite[Theorem 10]{BKV18} proved that the image has half dimension of the moduli stack of parahoric Higgs bundles. Our result actually can show that the image is indeed an affine space. We also need to point out that if $G$ is not assumed to be simply-connected, the Theorem \ref{thm:conclusion in the intro} does not hold.

We now turn to describe the geometry of generic fibers. The associated root stack actually play a role of bridge. Let $\omega_{\mathscr{C}}$ be the canonical line bundle over the root stack $\mathscr{C}$. A Higgs bundle over 
$\mathscr{C}$ determines an element of $H^{0}(\mathscr{C},\gc\times_{\mathbb{G}_m}\omega_{\mathscr{C}}^{\times})$, here $\omega_{\mathscr{C}}^{\times}$ is the complement of zero section of $\omega_{\mathscr{C}}$ and we treat it as a $\mathbb{G}_{m}$-torsor. We therefore call the  vector space $H^{0}(\mathscr{C},\gc\times_{\mathbb{G}_m}\omega_{\mathscr{C}}^{\times})$ the \emph{Hitchin base} and denote it by $\sA_{\mathscr{C}}$. Thus is defined the  Hitchin map:

\begin{equation}h_{\lambda}:\, \Hig_{\mathscr{C}}\rightarrow\sA_{\mathscr{C}}
\end{equation}

Similar to the curve case in \cite{DG02} and \cite{Ngo10}, for a generic point $a\in\sA_{\mathscr{C}}$, there is a smooth affine commutative group scheme $\sJ_{a}$ over $\mathscr{C}$, whose category of torsors $\Pic(\sJ_{a})$ acts simply and transitively on $h_{\lambda}^{-1}(a)$.  Another of our main results is:
\begin{theorem}
	For a generic point $a\in\sA_{\mathscr{C}}$, the fiber $h_{\lambda}^{-1}(a)$ is a gerbe banded by the Picard stack $\Pic(\sJ_{a})$, where $\sJ_{a}$ is a smooth affine commutative group scheme over $\mathscr{C}$. In particular, this implies that generic fibers of the parahoric Hitchin map are Abelian varieties (if we omit the automorphisms).
\end{theorem}
To say that a gerbe \emph{is banded by} $\Pic(\sJ_{a})$, means that there is an action of $\Pic(\sJ_{a})$ on $h_{\lambda}^{-1}(a)$ and for any object $C\in h_{\lambda}^{-1}(a)$, the following is an equivalence of categories:
\begin{equation}
P\in\Pic(\sJ_{a})\rightarrow P\cdot C\in h_{\lambda}^{-1}(a)
\end{equation}
This is a categorical formulation of saying that $\Pic(\sJ_{a})$ acts simply and transitively on $h_{\lambda}^{-1}(a)$.
In order that the reader can compare this with results in the literature, we state  this for parahoric Hitchin systems over $C$ in Theorem \ref{gen fiber}, but here we find it convenient to express this in terms of Hitchin systems over the root stack $\mathscr{C}$. These two statements are however equivalent. We give an application of this result to the $\SL_{n}$ case in the Corollary \ref{cor: app to SLn}. In \cite{SWW19}, such a result is proved by carefully studying the resolution of singularities of spectral curves along with some tricky techniques which are different from our arguments here.

Let us explain how this is related with earlier work.  Michael Groechenig  \cite{Groe16} studied parabolic vector bundles and parabolic Hitchin systems from the point view of Deligne-Mumford curves. Balaji and Seshadri \cite{BS14} showed how to obtain a parahoric $G$-bundle from a principal $G$-bundle over a Galois cover of $C$ endowed with an equivariant Galois action. The proof of our main Theorem \ref{Higgs equi} is motivated by their work. Our theorem that generic fibers are gerbes banded by the Picard stack generalizes a result of Donagi and Gaitsgory \cite{DG02} for "regularized Hitchin systems", who proved this over the complex field $\mathbb{C}$ and without parabolic structure. We should also point out that Ngo's treatment of the symmetry group of  generic fibers of Hitchin fibrations in \cite{Ngo10} also play an important role. This paper can also be viewed as a generalization of the main result of Su-Wang-Wen \cite{SWW19},  where the case of $\GL_{n}$ is carefully studied. Our  Picard stack, which is a tensor category of torsors over a smooth commutative group scheme on the curve $C$, might be regarded as the algebraic counterpart of the complete integrability of parahoric Hitchin maps over complex field $\mathbb{C}$, due to Baraglia, Kamgarpour and Varma \cite{BKV18}.

Our paper is organized as follows:
%
%

In Section 2, we recall the Hitchin fibration in non-parahoric case. The results there indicates what we should expect in parahoric cases. This mainly refers to work of Donagi and Gaitsgory \cite{DG02} and Ngo \cite{Ngo10}. 

%
In Section 3, we first recall the basics of a Bruhat-Tits system, with emphasis on the case of algebraic loop groups and their parahoric subgroups. This mainly makes references to \cite{Bour02},\cite{MT11}.  We also recall the construction of a Bruhat-Tits groups scheme $\sG_{\sF}$ over $C$  from parahoric data $(D,\sF)$ and the equivalence between the category of torsors over $\sG_{\sF}$ and the category of parahoric $G$-bundles which was set up in \cite{BS14}. Then we define parahoric Higgs bundles, Hitchin maps and the Hitchin base. We will talk about the parabolic case to provide a more intuitive picture.
%
%
%
%

In Section 4, we determine what the generic fiber of the parahoric Hitchin map is. The key input will be an equivalence of parahoric Higgs bundles on $C$ with Higgs bundles over root stack $\mathscr{C}$,  as we mentioned at the beginning of this introduction. And by this equivalence we can show that the generic fibers of Hitchin map are gerbes banded by Picard stacks. 
As an application, we prove that when $G=\GL_{n}$, the general fiber is the Picard variety  of a normalized spectral curves, thus  recovering one  of the main results of \cite{SWW19}.

In the Appendix, for readers' convenience, we briefly talk about the N\'eron model of finite type of an algebraic torus over a local field which is used in the Corollary \ref{cor: app to SLn}. And we prove the triviality of torsors of the N\'eron model over a complete DVR with algebraically closed field. Our proof essentially follows from a similar result by B\'egueri \cite[\S 3]{Be80}.
%


\noindent\textbf{Acknowledgements} The author thanks Eduard Looijenga for his great help not only in the mathematics but also in the writing. The author also wants to thank Peigen Li, Yichen Qin, Peng Shan, Xiaoyu Su, Xueqing Wen, Zhiwei Yun, Weizhe Zheng for their useful suggestions.

\section{Geometry of Hitchin System: the Non-Parabolic Case}\label{Section: non-para}

In this section, we briefly recall \emph{the} moduli stack of Higgs $G$-bundles over $C$, Hitchin maps and their geometric properties.
For our later purpose, we restrict ourselves to the case that $G$ is a semi-simple group over $k$ even though all the arguments, with the possible exception of Proposition \ref{Diamond}, hold when $G$ is reductive. As we talk about  non-parahoric cases, we need to assume that $g(C)\geq 2$ in this section.
 
\subsection{Chevalley maps and Kostant sections}
 The variety of Borel subgroups of G (the `flag variety' of $G$) is denoted $\sB$ and can be identified with $G/B$. Since a Borel subgroup is determined by its Lie algebra, we can identify $\sB$ with a closed orbit in the Grassmannian of $\g$. 

The Weyl group $N_G(T)/T$ is denoted $W$ as usual. It acts on $\gt$ as a Coxeter group.  As mentioned in the introduction, if $\chr(k)> 0$, then we assume it to be larger than twice the Coxeter number of every simple summand  of $\g$. Suppose for a moment that $\g$ is simple and let $h(\g)$ be its Coxeter number. The exponents $\{m_{i}\}_{i=1}^{n}$ of $\g$ are positive integers $\le h(\g)$ and there exists a system of basic W-invariants of $k[\gt]$ of degrees $\{m_{i}+1\}_{i=1}^{n}$. We know that $\#W$, is equal to the mapping degree of $\gt\rightarrow\gc$ which is $\prod_{i=1}^{n}(m_{i}+1)$. As a result, $p\nmid\#W$ and then $k[\gt]^W$ is a polynomial algebra by \cite{Bour02}. The natural map $k[\gt]^W\to k[\g]^{G}$ is an isomorphism of $k$-algebras, so that $\Spec k[\g]^{G}\simeq\Spec k[\g]^{W}$ is an affine space. 

The same conclusion holds if  $\g$ has several simple factors,  $\g_1, \dots, \g_r$, say, for then $k[\g]^G$ decomposes as a tensor product $k[\g_1]^{G_1}\otimes_k\cdots \otimes_k k[\g_r]^{G_r}$ and the $W$-action on $\gt$  decomposes accordingly as a direct sum. 
We write $\gc$  for $\Spec k[\gt]^{W}$ and sometimes refer to as the \emph{Chevalley base}.

\subsection{Higgs Bundles and Hitchin Maps}

We start with some basic knowledge of semisimple Lie algebras. 
If $\gb'\subset \g$ is a Borel algebra, then there exists a $g\in G$ such that $\Ad(g)$ takes $\gb'$ to $\gb$. This $g$ is unique up to left multiplication
with an element of $B$ and the postcomposite of this  map with $\gb\to \gb/\gn_\gb\cong \gt$ is independent of the choice of $g$. 
So the `universal Borel algebra' $\tilde{\g}:=\{(x,[\gb'])|x\in\gb',[\gb']\in\sB\}\subset \g\times\sB$ fits in a commutative diagram:
\begin{equation}\label{Gro Sim Res}
\xymatrix{
	\tilde{\g}\ar[r]^{\tilde{\chi}}\ar[d]^{\pi}&\gt\ar[d]^{\pi}\\
	\g\ar[r]^{\chi}&\gc}
\end{equation}    
The projection $\pi:\tilde{\g}\rightarrow\g$ is called \emph{Grothendieck's simultaneous resolution}.
If we restrict this diagram to the regular part $\g^{reg}$ of $\g$, we get a Cartesian diagram:
\begin{equation}\label{car reg}
\xymatrix{
	\tilde{\g}^{reg}\ar[r]^{\tilde{\chi}^{reg}}\ar[d]^{\pi^{reg}}&\gt\ar[d]^{\pi}\\
	\g^{reg}\ar[r]^{\chi^{reg}}&\gc}
\end{equation}

There is a section of the Chevalley map $\chi^{reg}:\g^{reg}\rightarrow \gc$, and we explain it now. Let $\sR(G,T)$ be the set of roots with respect $T$ so that we have the following root space decomposition:
\begin{equation*}
\g=\gt\oplus(\oplus_{\alpha\in \sR(G,T)}\g_{\alpha})
\end{equation*}
We define a regular nilpotent element: 
\begin{equation*}
x^{+}=\sum_{i=1}^{r}e_{\alpha_{i}}
\end{equation*}
where $e_{\alpha_{i}}\in\g_{\alpha_{i}}\ssm\{0\}$. Because of Assumption \ref{blas}, the Jacobson-Morozov lemma holds for $\g$ and so
$x^+$ is part of a $\mathfrak{sl}_{2}$ triple $(x^{+},x^{-},h)$. The choice of this triple is unique up to conjugation by an element of $Z_{G}(x^{+})$.

The following theorem is well-known and due to Kostant \cite{Kos63} over $\mathbb{C}$, and due to Veldkamp \cite{Vel72} when $p\nmid\# W$. 
\begin{theorem}\cite[Proposition 6.3]{Vel72}
	The map $\chi^{reg}$ is a smooth map. Given a regular nilpotent element $x^{+}$, then $\chi^{reg}: x^{+}+Z_{\g}(x^{-})\rightarrow \gc$ is an isomorphism. 	
\end{theorem}

It is worth pointing out that Riche \cite[Theorem 3.3.2]{Ric17} proved this with weaker assumption on $p$, where $Z_{\g}(x^{-})$ is replaced by any $\mathbb{G}_{m}$ invariant complement of $[x^+,\gn]$ containing in $\gb$. The inverse of this isomorphism is called a \emph{Kostant section}. Clearly,  the Kostant section constructed above only depends on the choice of the $\mathfrak{sl}_{2}$-triple. It takes $0$ to  $x^+$ and lands  in $\g^{reg}$. Two Kostant sections that take $0$ to $x^+$ are $Z_{G}(x^{+})$-conjugate. We shall denote such a Kostant section by $\epsilon$.

For every root $\alpha\in \sR(G,T)$, we regard its differential  $d\alpha$ as a linear form on $\gt$, so that we can form the product
\begin{equation}
\textstyle \prod_{\alpha\in \sR(G,T)}d\alpha\in k[\gt].
\end{equation}
This element is obviously $W$-invariant and so defines an element of $k[\gc]$. 
\begin{definition}[Discriminant Divisor]
	The resulting principal divisor $\mathscr{D}$ in $\gc$ is called the \emph{discriminant divisor}. 
\end{definition}

We denote its complement in $\gc$ by $\gc^{rs}$. The following lemma explains the name:

\begin{lemma}\cite{Ngo10}
	The discriminant divisor $\mathscr{D}$ is reduced and the natural map $\pi:\gt\rightarrow\gc$ is etale over $\gc^{rs}$.
\end{lemma}

We note that in  the restriction of  the Cartesian diagram \ref{car reg} over $\gc^{rs}$,
\begin{equation}\label{rs cartesian}
\xymatrix{
	\tilde{\g}^{rs}\ar[r]^{\tilde{\chi}^{rs}}\ar[d]^{\pi^{rs}}&\gt^{rs}\ar[d]^{\pi}\\
	\g^{rs}\ar[r]^{\chi^{rs}}&\gc^{rs}
}
\end{equation}
both the left and right vertical arrows are $W$-torsors. 
We can now give the definition of a Higgs bundle and the Hitchin map. Naively speaking, we can view a Hitchin map as a family of diagrams \ref{Gro Sim Res}  twisted by very ample line bundles over $C$. Given any variety $V$ with an $\mathbb{G}_{m}$-action and a line bundle $\sL$ over $C$, we shall denote the fiber bundle $\sL^{\times}\times_{\mathbb{G}_{m}}\! V$ over $C$ (with fiber $V$) by $V_{\sL}$, here $\sL^{\times}$ is the complement of zero section in total space of $\sL$.


\begin{definition}[$G$-Higgs Bundles]
	Let $\sL$ be a line bundle over $C$ and we assume $\deg\sL>2g$, or $\sL=\omega_{C}$, the canonical line bundle. 
	A \emph{Higgs bundle over $C$} is a pair $(E,\theta)$, where $E$ is a principal $G$-bundle and
	\begin{equation}
	\theta\in H^{0}(C,E(\g)\otimes\sL)
	\end{equation} 
	(called a \emph{Higgs field}). Here $E(\g)$ is the adjoint bundle. 
	
	If $\theta$  is  everywhere regular, i.e. $\theta$ lies in $E(\g^{reg})_{\sL}$, we call $(E,\theta)$ a \emph{regular Higgs bundle}. 
\end{definition}

The following theorem is known: 
\begin{theorem}
	The moduli functor 
	\begin{equation}
	\mathbf{H}: \mathsf{Sch}/k\rightarrow \mathsf{Groupoid}
	\end{equation}
	assigning to a $k$-scheme $Y$ the groupoid of Higgs bundles over $Y\times C$ is represented by an algebraic stack  $\Lhig$. This stack contains an open substack $\Lhig^{reg}$ which parametrizes regular Higgs bundles and this substack is smooth.
\end{theorem}

Since $\mathbb{G}_{m}$ acts on $\g$ as scalar multiplication and the action commutes with the adjoint action of $G$, we can form a smooth stack $(\g\otimes_k\sL)/G$ over $C$.

\begin{lemma}
	The algebraic stack $\Lhig$ is equivalent to a Hom stack:
	\begin{equation}
	\Hom_{C}(C,[\g\otimes\sL/G])
	\end{equation}
\end{lemma}
\begin{proof}
	
	By definition, $\Hom(C,[\g\otimes\sL/G])$ assigns to $Y\in \mathsf{Sch}/k$ the set of pairs 	
	$(\sE,\Theta)$, where $\sE$ is a $G$-bundle over $Y\times C$ and 
	$\Theta:\sE\rightarrow \mathsf{pr}_C^{*}(\g\otimes{\sL})$ is a $G$-equivariant map ($\mathsf{pr}_C: Y\times C\rightarrow C$ is the projection).
	It is not difficult to see that $\Theta$ is then an element of 
	$H^{0}(Y\times C,\sE(\g)\otimes\mathsf{pr}_C^{*}\sL)$.
\end{proof}
\begin{remark}
	Our description is a little bit different from that in \cite{Ngo10}, where Ngo gave a more general result.
\end{remark}

Because of our Assumption \ref{blas}, $k[\gt]^W$  admits algebraically independent  generators homogeneous of degree $1+m_1, \dots, 1+m_n$, 
where $m_1, \dots, m_n$ are the exponents of $W$. Then $\mathbb{G}_{m}$ acts on $\gc$ with weights $1+m_1, \dots, 1+m_n$, and this identifies $\gc_{\sL}$ with $\oplus_{i=1}^n\sL^{\otimes{(1+m_i)}}$. The natural $G$-invariant map $\g\rightarrow\gc$ is also $\mathbb{G}_{m}$-equivariant and so there is an induced map 
$E(\g)\otimes\sL\to \gc\otimes\sL$. We shall write $\sA$ for $H^{0}(C,\gc\otimes_{\sO_{C}}\sL)$.  So by our discussion, $\sA$ can be identified with the affine space
$\oplus_{i=1}^n H^0(C, \sL^{\otimes{(1+m_i)}})$.

\begin{definition}[Hitchin Maps]
	The \emph{Hitchin map} is defined as:
	\begin{equation}
	h:\Lhig\rightarrow \sA=H^{0}(C,\gc_{\sL}), \quad 
	(E_{G},\theta)\mapsto \chi(\theta).
	\end{equation}   	    	
\end{definition}

A Kostant section defines a section of $E_G(\g)\otimes\sL\to \gc\otimes\sL$ which takes values in $E_{G}(\g^{reg})_{\sL}$. This shows that $h:\Lhig_G^{reg}\rightarrow \sA$ is surjective.

To describe the generic fibers of the Hitchin map $h$, we need the following theorem in \cite{Ngo10}, which will also be useful in \emph{parahoric} case. Let us first introduce the closed subscheme of $\g\times G$ defined by
\begin{equation}
\sI:=\{(x,g)\in \g\times G\, |\, Ad_{g}(x)=x\}.
\end{equation}
Note that $\sI$  is $G$-invariant for the adjoint resp.\ conjugate  action of $G$ on $\g$ resp.\ $G$  (so given by  $h(x,g)=(\Ad_h(x), hgh^{-1})$) and that this makes $\sI$ a group scheme over $\g$, the fiber over $x\in \g$ being $Z_G(x)$. The $\mathbb{G}_{m}$-action on the first factor of $\g\times G$ evidently preserves $\sI$.

\begin{propdef}\cite[Lemma 2.1.1]{Ngo10}\label{Univ cen}
	There exists a smooth affine commutative group scheme $\sJ$ over $\gc$ to which the (scalar) $\mathbb{G}_{m}$-action on $\gc$ naturally lifts and a natural morphism  $\phi:\chi^{*}(\sJ)\to \sI$  of  group schemes over $\g$ with $\mathbb{G}_{m}$-action with the property that this is an isomorphism over $\g^{reg}$. 	We call the smooth commutative group scheme $\sJ$ over $\gc$ the \emph{universal regular centralizer}.
\end{propdef}

Observe that  $\sJ\times_{\mathbb{G}_{m}}\!\sL^{\times}$ is a commutative smooth group scheme over $\gc_\sL$.  Hence for every  $a\in \sA$, i.e.,  section $a:C\rightarrow \gc_{\sL}$, the pull-back 
\begin{equation}
\sJ_{a}:=a^{*}(\sJ\times_{\mathbb{G}_{m}}\!\sL)
\end{equation}
is a commutative group scheme over $C$.

As $\gt_{\sL}\rightarrow\gc_{\sL}$ is $W$-covering, similarly to the construction of $\sJ_{a}$, we can define so-called cameral cover.

\begin{definition}[cameral cover]\label{def:cam cover}
	Given a point $a\in\sA$, the \emph{cameral cover} $\tilde{C}_{a}$ of $C$ is defined as follows:
	\begin{equation}
	\tilde{C}_{a}:=C\times_{\gc_{\sL}}\gt_{\sL}
	\end{equation}
\end{definition}

As $\mathscr{D}$ is $\mathbb{G}_{m}$-invariant, so is its complement. Then $\mathscr{D}_{\sL}$ is a divisor of $\gc_{\sL}$ and we denote its complement in $\gc_{\sL}$ by $\gc_{\sL}^{rs}$.

\begin{lemma}
	Let $\sA^{\heartsuit}\subset \sA$ be the subset of $a\in\sA$ such that image of the generic point of $C$ by $a$ is contained in $\gc^{rs}_{\sL}$. Then $\sA^{\heartsuit}$ is an open subvariety of $\sA$. For every $a\in\sA^{\heartsuit}$, the corresponding cameral curve $\tilde{C}_{a}$ is generically a $W$-torsor over $C$ and is reduced.
\end{lemma}
\begin{proof}
	The first assertion follows from the fact that $\gc^{rs}$ is open in $\gc$.
	The Kostant section shows that for every $a\in\sA^{\heartsuit}$, the preimage $h^{-1}(a)$ is non-empty.
	
	Consider the following Cartesian diagram:
	\begin{equation}
	\xymatrix{
		\tilde{\g}^{rs}\ar[r]^{\tilde{\chi}}\ar[d]^{\pi^{reg}}&\gt^{rs}\ar[d]^{\pi}&\\
		\g^{rs}\ar[r]^{\chi}&\gc^{rs}
	}
	\end{equation}
	Let $U$ be an open subset of $C$ and $a:U\rightarrow \gc^{rs}$, then:
	\begin{equation}
	\tilde{U}_{a}:=U\times_{\gc^{rs}_{\sL}}\gt^{rs}_{\sL}
	\end{equation}
	We know that $\tilde{U}_{a}$ is a $W$-torsor over $U$, thus $\tilde{U}$ is reduced. Since $\gt_{\sL}\rightarrow\gc_{\sL}$ is flat, we know that $\tilde{C}_{a}\rightarrow C$ is also flat. Thus $\tilde{U}_{a}$ is open and dense in $\tilde{C}_{a}$. As a result, $\tilde{C}_{a}$ is reduced. 
\end{proof}


To study generic fibers of Hitchin maps, we usually focus on fibers over certain open subvarieties of $\sA$. In the following we will introduce some of them following Ngo's definitions in \cite{Ngo10}. Since $\mathscr{D}$ is a $\mathbb{G}_{m}$-invariant divisor  in $\gc$, $\mathscr{D}_{\sL}$ is a divisor in $\gc_{\sL}$. Similarly, $\mathscr{D}^{sing}_{\sL}$ is a closed subvariety of $\mathscr{D}_{\sL}$, here $\mathscr{D}^{sing}$ is the subvariety of singular points of $\mathscr{D}$. 

\begin{definition}
	We say $a$ maps $C$ transversely with the reduced divisor $\mathscr{D}_{\sL}$ in $\gc_{\sL}$ if $a^*\mathscr{D}_{\sL}$ is a reduced divisor. This in particular implies the following:
	\begin{itemize}
		\item[(a)] $a$ maps generic point of $C$ into $\gc^{rs}$
		\item[(b)] if for some closed point $y\in C$, $a(y)\in \mathscr{D}$, then $a(y)$ does not lie in $\mathscr{D}^{sing}$, and the intersection number of $a(V)$ and $\mathscr{D}^{sm}$ is one, for a small neighbourhood of $y$.
	\end{itemize}
\end{definition}

\begin{proposition}\label{Diamond}
	Recall that we assume $\deg\sL>2g$, or $\sL=\omega_{C}$. Then the set of $a\in \sA$ which are transversal to $\mathscr{D}_\sL$ make up
	a nonempty open subset $\sA^{\Diamond}$ of $\sA^{\heartsuit}$.  
	For $a\in\sA^{\heartsuit}$, the cameral cover $\tilde{C}_{a}$ is smooth if and only if $a\in\sA^{\Diamond}$.
\end{proposition}

This proposition is proved by Ngo in \cite{Ngo10} for $\deg \sL>2g$, where it is divided into Proposition 4.7.1 and Lemma 4.7.3. Our proof is a direct application of arguments there.

\begin{proof}   	
	To prove  that $\sA^{\Diamond}$ is nonempty,  we follow Ngo's proof: we only need to check that the following map is surjective for any closed point $y\in C$:
	\begin{equation}
	H^{0}(C,\gc_{\sL})\rightarrow \gc_\sL\otimes\sO_{C}/m_{y}^{2}
	\end{equation}
	Let $m_1\le m_2\le \cdots \le m_n$ be the exponents of $G$. Since $m_1>0$,  the surjectivity follows from the following fact that 
	$\gc_{\sL}\simeq\oplus_{i=1}^{n}\sL^{\otimes(1+m_{i})}$.
	
	The second part is a local analysis, it also holds when $\sL=\omega_{C}$ by Ngo's proof.
\end{proof}
 	
\begin{theorem}\cite{DG02}\cite{Ngo10}\label{non-para generic fiber}
	The restriction of $h$ to $\Lhig_G^{reg}$ is smooth, with each fiber $h^{-1}(a)\cap\Lhig_G^{reg}$ a gerbe banded by Picard stack $\Pic(\sJ_{a})$. Moreover $\Lhig_G^{reg}$ coincides with $\Lhig_G$ over $\sA^{\Diamond}$.
\end{theorem} 
%
This theorem follows from Proposition 4.3.3, Proposition 4.3.5 and Proposition 4.7.7 in \cite{Ngo10} where it is assumed that $\deg\sL>2g$, but in our case the canonical line bundle also works. The second assertion can also be deduced from Donagi and Gaitsgory \cite[Corollary 17.6]{DG02}. We will discuss the case $G=\SL_{n}$ to provide a clear picture.
\begin{example}    	 
	When $G=\SL_{n}$, the concept of `cameral curve' can be replaced by `spectral curve', which is more intuitive. For $a\in\sA$, it determines the characteristic polynomial of $\theta$. Thus $a\in\sA$ defines a Cartier divisor of the surface $\sL$, where we treat the total space of $\sL$ as a surface. We may denote the divisor by $C_{a}$.  If $a\in\sA^{\Diamond}$, the corresponding spectral curve $\pi:C_{a}\rightarrow C$ is smooth, and the pair $(E,\theta)$ can be identified with a line bundle $\sM$ over $C_{a}$ (we refer  to \cite{BNR} for more details). Then the Higgs field $\theta$ will come from the $\sO_{C_{a}}$ module structure of $\sM$. By the local description, we know that $\theta$ is regular everywhere!  In this case $\sJ_{a}$  is exactly the group scheme $\pi_{*}\sO_{C_{a}}^{\times}$ over $C$. In  this case, the above theorem asserts that $\Pic^{0}(C_{a})$ acts simply and transitively on $h^{-1}(a)$.
\end{example}

In this section, we gave a quick introduction of Hitchin systems. Theorem \ref{non-para generic fiber} provides a formal description of generics fibers of Hitchin systems. In order to generalize that theorem  to the parahoric case, we need to introduce Bruhat-Tits group schemes over $C$. This 
will in particular make Heinloth's uniformization theorem \cite{Hein10} available to us. This is the subject of the next section.

\section{Parahoric Higgs Bundles}\label{Section: para higgs}

In this section, we will first talk about \emph{Bruhat-Tits} group schemes over $C$ which is constructed via parahoric data. To be more precise, given a parahoric data $(D,\sF)$, we construct a smooth affine group scheme $\sG_{\sF}$ over $C$. This is a special case of the monumental work of Bruhat and Tits \cite{BT72}, \cite{Tits79},\cite{BT84}. But for readers' convenience, we include a brief introduction of Tits systems, Tits buildings leading to the Proposition \ref{prop:conj of para and alcove} which we actually use to define the so-called \emph{parahoric data}. 

Then we will introduce the definition of $\sG_{\sF}$-Higgs bundles, and show that the moduli functor of $\sG_{\sF}$-Higgs bundles over the curve $C$ is represented by an algebraic stack. And we will also introduce the Hitchin map. 

We should mention the pioneering work of Bhosle and Ramanathan \cite{BR89}, where they studied the moduli of stable parabolic principal bundles over Riemann surfaces and of  Balaji and Seshadri \cite{BS14} who studied parahoric principal bundles over Riemann surfaces.

As always, $C$ is a smooth projective curve over an algebraically closed field $k$ such that $2g(C)-2+\deg D>0$. In this section,
\emph{we assume $G$ to be a semi-simple simply connected algebraic group over $k$}. We fix a maximal torus $T\subset G$, and a Borel subgroup $B$ containing $T$. 
\subsection{Bruhat-Tits Group Scheme over $\sO$}\label{section:BToverO}

We first recall what a Tits system is for an abstract group $G$, we refer to Chapter IV in \cite{Bour02}, or in Chapter 11 in \cite{MT11}. 

\subsubsection{\textbf{Tits System}}\label{Section: T Sys}

We do not need to assume $G$ to be an algebraic group for the present. 
\begin{definition}[Tits system]\label{Tits sys}
	A \emph{Tits  system} is a triple $(G,B,N)$, where $G$ is a group, $B,N$ are subgroups of $G$ satisfying the following conditions:
	\begin{enumerate}
		\item[(T1)] $B\cup N$ generates $G$ and $B\cap N$ is a normal subgroup of $N$;
		\item[(T2)] $W=N/(B\cap N)$ is generated by a system of  involutive elements which we denote by $S$;
		\item[(T3)] $sBw\subset BwB\cup BswB$ for any $s\in S, w\in W$;
		\item[(T4)] for $s\in S$, $sBs^{-1}\nsubseteq B$.
	\end{enumerate}    	
\end{definition}
We shall write  $T$ for $B\cap N$ and $W$ for $N/B\cap N$. We refer to $W$ as the \emph{Weyl group} of the Tits system.
For any $w\in W$, we write $C(w)$ for the double coset $BwB$. It is clear that for $w, w'\in W$, 
$C(ww')\subset C(w)\cdot C(w')$. Notice that (T3) is equivalent to 
\begin{equation}
C(s)\cdot C(w)\subset C(w)\cup C(sw)
\end{equation}

Now we recall the following theorem about  Coxeter subsystems:
\begin{theorem}\label{rel in Cox}
	The set $S$ is a minimal set of generators of $W$ and if $X, X'\subset S$, then the pair $(W_{X},X)$ is a Coxeter system and
	$W_{X\cap X'}=W_X\cap W_{X'}$. In particular, $W_{X}\subset W_{X'}$ if and only if $X\subset X'$. 
\end{theorem}

For $X\subset S$, we put $G_{X}:=BW_{X}B=\cup_{x\in X} C(x)$. This is easily seen to be a subgroup of $G$.

\begin{theorem}\label{bij in Tits sys} The passage from a subset $X$ of $S$ to the subgroup $G_X$ of $G$ defines a bijection 
	between the subsets of $S$ and the subgroups in $G$ containing $B$.  We have  $G_{X\cap X'}=G_X\cap G_{X'}$, in particular, $G_{X}\subset G_{X'}$ if and only if $X\subset X'$.
	
	An element  $w\in W$ lies in $S$ if and only if $B\cup C(w)$ is a subgroup of $G$. 
\end{theorem}

The last assertion of  Theorem \ref{bij in Tits sys} shows that $S$ is determined by $(B,N)$, and  that $(G_{X},B,N_{X})$ is 
a Tits subsystem.

\begin{definition}
	A subgroup $P\subset G$ is called a \emph{parabolic subgroup} if it contains a conjugate of  $B$.
\end{definition}

There is a close relation between conjugacy classes of parabolic subgroups with sub-Coxeter systems.
\begin{theorem}\cite{MT11}\label{para} For a Tits system  $(G,B,N)$ the following holds.
	\begin{enumerate}
		\item[(i)] A subgroup $P$ of $G$ is a parabolic group, if and only if it is conjugate to $G_{X}$ for some $X\subset S$. Moreover $X$ is then unique.
		\item[(ii)] Let $Q$ and $Q'$  be subgroups of a parabolic subgroup $P$ of $G$ and let  $g\in G$ be such that $gQg^{-1}\subset Q'$. Then $g\in P$. In particular, $P$ is its own normalizer.
	\end{enumerate}
\end{theorem}

\subsubsection{\textbf{The Tits building}}
The collection of proper parabolic subgroups of $G$ makes up a POset(PO for partial order). The \emph{Tits building} $\sB(G,B,N)$ associated to this Tits system is defined when $S$ is finite: if $n:=\#(S)$, then $\sB(G,B,N)$ is the simplicial complex whose barycentric subdivision realizes this POset: If $P\subset G$ is a proper parabolic subgroup conjugate to $G_X$, then it defines a simplex $\Delta_P$ of $\sB(G,B,N)$ of dimension $\#(X)=n-k$. This is inclusion-reversing: if $Q\subset G$ is a parabolic subgroup contained in $P$, then $\Delta_P$ is a facet of $\Delta_Q$. So $\Delta_B$ defines $(n-1)$-simplex (the fundamental  simplex), and if $s\in S$, then 
$G_s$ defines a codimension one face of this chamber and $G_{S\ssm \{ s\}}$ a vertex of it. It is clear that $\sB(G,B,N)$ comes with a $G$-action.

The proper parabolic subgroups of $G$ that contain $T$ are the $N$-conjugates of the standard proper parabolic subgroups. They define a subcomplex
of  $\sB(G,B,N)$ that is of course $N$-invariant. Note however that $N$ acts on this subcomplex via $W=N/T$. Indeed, this is the Coxeter complex of the pair $(W,S)$ that is defined in a similar manner as the Tits complex (by taking the $W$-conjugates of the proper standard Weyl subgroups $W_X\subsetneq W$). It is a combinatorial $(n-1)$-sphere when $W$ is finite and it is a combinatorial real affine $(n-1)$-space when the Tits system is of the type that we discuss in the next section (where $G$ is a reductive group over a local field). A $G$-translate of this subcomplex is called an \emph{apartment} of $\sB(G,B,N)$.

\subsubsection{\textbf{Conjugacy Classes of Parahoric Subgroups}}
We now assume $G$  is a semi-simple and simply connected group defined over $k$. Let $\sO$ be a DVR  isomorphic to $k[[t]]$ whose maximal ideal resp.\ field of fractions is denoted $m$ resp.\ 
$\sK$  (so $\sK$ is  isomorphic to $k((t))$). It is known that $\hat{G}:=G(\sK)$ is the set of $k$-rational points of the algebraic loop group associated with $G$, but for the moment we prefer to treat it as an abstract group. We recall how this group can be endowed  with a Tits system that makes use of the subgroup $G(\sO)$ and the evaluation map
\begin{equation}
ev:G(\sO)\rightarrow G(k)
\end{equation}
given by reduction modulo the maximal ideal. We shall compare this with the Tits system $(G, B, N)$ described above.

\begin{theorem}\cite{Iwa66}\label{Loop Tits system}
	Let $\hat{I}$  be the preimage of $B(k)$ under $ev$ and $\hat{N}:=N_{G}(T)(\sK)$. Then the triple $(\hat{G},\hat{I},\hat N)$ forms a Tits system.
\end{theorem}

The group $\hat{I}$ is called  a \emph{Iwahori subgroup} of $\hat G$. We put $\hat{W}:=\hat{N}/(\hat{I}\cap\hat{N})$ and  explain how to choose a system of generators $\hat{S}$ of $\hat{W}$.    
We first verify that $\hat{I}\cap\hat{N}=T(\sO)$. As $k$ is algebraically closed, $G$ splits over $k$ and then $\hat{N}/T(\sK)=W$. Then we can see $\hat{N}\cap B(\sK)=T(\sK)$ because we may choose representative of cosets $\hat{N}/T(\sK)$ from $G$. It follows that $\hat{I}\cap\hat{N}=T(\sO)$.

This implies  that we have the exact sequence:
\begin{equation*}
1\rightarrow T(\sK)/T(\sO)\rightarrow \hat{N}/(\hat{I}\cap\hat{N})\rightarrow N_{G}(T)/T\rightarrow 1.
\end{equation*}
There is a natural isomorphism $X_*(T)\to T(\sK)/T(\sO)$. 
%
So the above exact sequence boils down to
\begin{equation}\label{affine weyl}
1\rightarrow X_*(T)\rightarrow \hat{W} \rightarrow W\rightarrow 1.
\end{equation}
A splitting of this sequence is induced by the inclusions $B(k)\subset \hat{I}$ and $N_G(k)\subset N_G(\sK)$ as this  induces a group homomorphism $W=N_G(k)/(B(k)\cap N_G(k))\to  \hat{N}/(\hat{I}\cap\hat{N})=\hat W$. This makes $\hat{W}$ a semi-direct product 
$W\ltimes  X_*(T)$ which is faithfully represented on $X_*(T)\otimes\BR$  by affine-linear transformations. 
Then $X_{*}(T)\otimes\mathbb{R}$ endowed with the $\hat W$-action is a geometric realization of the apartment that the Tits theory associates to $\hat I\cap \hat N=T(\sO)$. 
\begin{remark}
	Iwahori \cite{Iwa66} showed that if $G$ is not simply connected, then $(\hat{G},\hat{N},\hat{I})$ forms what is called a \emph{generalized Tits system}. The exact sequence \ref{affine weyl} still holds, but $\hat W$ need no longer be a Coxeter group; rather it is an extension of an affine Coxeter group by the (finite) fundamental group of $G$, which can be easily deduced from the exact sequence \ref{affine weyl}.  We still have a Tits complex, where the role of a standard parabolic subgroups is taken by the neutral component of the groups $\hat G_X$.		
\end{remark}  

\begin{upshot}
	By Theorem \ref{para}, to determine the conjugacy classes of parabolic subgroups, it is enough to understand the subset of generators of $\hat{W}$.
\end{upshot}

Let us identify the basic chamber and the generating set $\hat S$ of $\hat W$ in these terms. Since $G$ is semisimple simply connected, so without loss of generality, we now assume that $G$ is almost simple, in other words, that $\sR(G,T)$ is an irreducible root system.
We regard each root as an affine-linear form on $X_{*}(T)\otimes\mathbb{R}$.  We denote by $\alpha_{\max}$ the highest root of $T$ in $B$
(so with respect to the system of simple roots $\Delta$) and put $\alpha_0:=1- \alpha_{\max}$. 
Then $\hat \Delta:=\{\alpha_0\}\cup \Delta=\{ \alpha_i\}_{i=0}^n$ consists of $n+1$ affine-linearly independent forms on $X_{*}(T)\otimes\mathbb{R}$, so that  the intersection of the half spaces defined by $\alpha\ge 0$, $\alpha\in\hat\Delta$, is a  geometric simplex  in $X_{*}(T)\otimes\mathbb{R}$ with supporting hyperplanes defined by the affine-linear forms $\alpha\in\hat\Delta$. We refer to this simplex as the \emph{fundamental alcove} and denote it by $\Al$. To be more precise:
\[
\Al:=\{\lambda\in X_{*}(T)\otimes\BR|\langle\lambda,\alpha_{i}\rangle, i=0,1,\ldots,n\}
\]
It corresponds to the basic simplex of the apartment.  If $s_{0}$ denotes the affine reflection with respect the affine hyperplane defined by $\alpha_0=0$: $x\mapsto x+\alpha_0(x)\alpha^\vee_{\max}$, then $\hat{S}:=\{s_0\}\cup S=\{s_i\}_{i=0}^n$ is the generating set of $\hat{W}$ produced
by the Tits theory. To conclude:

\begin{proposition}\cite{BT84}\label{prop:conj of para and alcove}
	The affine-linear independence of $\alpha_0, \dots,\alpha_n$ implies that by assigning to each facet $F$ of $\Al$ the subset $\emph{X(F)}=\{s_{i}\in S|\,\,s_{i}|_{F}=Id\}$ of $\hat{S}$, we obtain a  bijection between the facets of $\Al$ and the \emph{proper} subsets of $\hat{S}$. In particular, we have a bijection:
	\[
	\{\text{conjugacy classes of parahoric subgroups}.\}\leftrightarrow\{\text{facets of the fundamental alcove} \; \Al.\}
	\]
\end{proposition}

We noted that $X_*(T)\otimes \BR$ is the geometric realization of an apartment. Its decomposition into relatively open facets is defined by the affine hyperplane arrangement  whose hyperplanes are defined by $\langle \alpha, \lambda\rangle=\ell$, where $\alpha\in \sR(T, G)$ and $\ell\in\BZ$. So 
$\lambda, \lambda'\in X_*(T)\otimes \BR$ belong to the same relatively open facet $F$ if and only if for all $\alpha\in \sR(T, G)$, 
$\lfloor \langle \alpha, \lambda\rangle\rfloor=\lfloor \langle \alpha, \lambda'\rangle\rfloor$ (use that $-\alpha$ is also a root).

For a facet $F$ of $X_*(T)\otimes \BR$,  this common value of the $\lfloor \langle \alpha, \lambda\rangle\rfloor $ for all $\lambda$ in the relative interior of $F$ is of course  $\lfloor\min\{\alpha|F\} \rfloor$. The following Lemma describes the parahoric subgroup associated to $F$ in terms of these integers.
For a root $\alpha\in \sR(G, T)$, denote by $U_\alpha\subset G$ the associated root subgroup (a copy of $\mathbb{G}_a$).

\begin{lemma}\label{lemma:structure of parahoric} 
	For a facet $F$ of $X_*(T)\otimes \BR$, the associated parahoric subgroup $\hat{P}_F\subset G(\sK)$
	is  generated by $T(\sO)$ and the $U_{\alpha}(m^{-\lfloor \min\{ \alpha|F\}\rfloor})$ with $\alpha\in \sR(G,T)$.
\end{lemma}	 

For a proof,  see for example Tits  \cite{Tits79}. In the following example, we focus on standard parahoric subgroups, whose structures are more explicit.

\begin{example}\label{example:parabolic}
	Let $F$ be a face of the fundamental alcove $\Al$ which contains the origin. If  $X\subset \Delta$ denotes the set simple roots that are zero on $F$,
	then  $F$ is defined by the following (in)equalities: $\alpha =0$ for $\alpha\in X$, $\alpha \ge 0$ if $\alpha\in \Delta\ssm X$ and $\alpha_{\max}< 1$.
	Hence for every  root $\alpha\in \sR(G,T)$ we have $\min\{\alpha|F\}\in [-1,0]$. This minimum is $0$ precisely when $\alpha$ is positive or lies
	in the root subsystem $\sR_X:=W_X(X)$; for all other roots, this minimum lies in $[-1, 0)$. It follows that $\hat{P}_F$ is generated by 
	$T(\sO)$ and the $U_{\alpha}(\sO)$ with $\alpha\in \sR_+(G,T)\cup \sR_X=\sR(P_X,T)$ and the $U_{\alpha}(m)$ for all other roots. So this is just $ev^{-1} P_X$ as defined before.
\end{example}


\subsubsection{\textbf{Construction of Bruhat-Tits Group Schemes}}
A subgroup of $\hat G$ is parabolic with respect to this Tits system, if it contains a conjugate of $\hat{I}$. In this particular situation, it is however customary to call such a group \emph{parahoric} (a contamination of Iwahori and parabolic). A parahoric subgroup corresponds to a simplex of the Tits building. By Theorem \ref{para}, a parahoric subgroup is conjugate to some $\hat{G}_{X}=\cup_{w\in \hat{W}_{X}}\hat{I}w\hat{I}$ for a unique subset $X\subset \hat S$. If $s_{0}\notin X$ (so that $X\subset S$), then (as is described in \cite{BS14}),  $\hat{G}_{X}$ is the preimage of $G_{X}$ under the map $ev:G(\sO)\rightarrow G(k)$. From the preceding, we conclude that by assigning to a facet $F$ of $\Al$, the standard parahoric subgroup $\hat G_X$ associated with the corresponding subset $X\subset\hat S$, we obtain a bijection between the facets of $\Al$ and the conjugacy classes of the \emph{proper} parahoric subgroups of $\hat{G}$. We then write $\hat P_F$ for $\hat G_X$. The following theorem is crucial for our purpose:

%

\begin{theorem}\cite{BT84}\label{BT existence}
	For every facet $F$ of the Tits building $\sB(\hat G, \hat I, \hat N)$, there exists a unique smooth affine group scheme $\sG_{F}$ over $\sO$, such that $\sG_F(\sO)=\hat{P}_{F}$ as subgroups of $\hat{G}$ so that an inclusion $F\subset F'$ defines an $\sO$-morphism $\sG_{F'}\to \sG_{F}$
	which on the $\sO$-valued points induces the injection $\hat P_F'\subset \hat P_F$. If 
	$F$ is the origin of the fundamental alcove, then this group is obtained from $F$ by the base change defined by $k\subset \sO$.
\end{theorem}

Since our focus of the thesis is on moduli of bundles over curves, we will use the following proposition (Lemma  3.18 and Lemma 3.19 in  Chernousov, Gille and Pianzola \cite{CGP12} which is  attributed by the authors to Raghunathan and Ramanathan):

\begin{proposition}\label{grp over curve}
	Let $C$ be a smooth projective curve and $x\in C$ a closed point. Put $U:=C\smallsetminus\{ x\}$, write $\sO_x$ for the formal completion of  $\sO_{C, x}$  and 
	$\sK_x$ for the field of fractions of $\sO_x$.
	\begin{itemize}
		\item[(i)](Glueing property) Assume we have a triple $(\sG_{U},\sG_{\sO_{x}},f)$, with $\sG_U$an affine group scheme over $k[U]$ of finite type, $\sG_{\sO_x}$  an affine finitely presented group 
		over $\sO_x$ and $f$ an isomorphism of  $\sK_x$-group schemes between the pull-backs of $G_U$ and $G_{\sO_x}$ over 
		$\Spec(\sK_x)$. 
		Then there exists an affine group scheme of finite type $\sG$ over $C$  and isomorphisms $\sG|_U\simeq \sG_U$ and $\sG|_{\sO_x}\simeq \sG_{\sO_x}$ that are compatible with $f$. If both $\sG_U$ and $\sG_{\sO_x}$ are smooth, then so is $\sG$. This group $\sG$ is unique up to isomorphism;
		\item [(ii)](Functorial property) Assume we have $(\sG_{U},\sG_{\sO_x},f)$ and $(\sG_{U}',\sG_{\sO_x}',f')$ similar with above, and we have group homomorphisms $\phi_{U}:\sG_{U}\rightarrow\sG'_{U}$, $\phi_{x}:\sG_{\sO_x}\rightarrow\sG_{\sO_x}'$ compatible with $f,f'$, then there is a unique homorphism $\phi:\sG\rightarrow\sG'$ which glues $\phi_{U}$ and $\phi_{x}$;
	\end{itemize}
\end{proposition}

The second part of the Proposition is slightly stronger than that in \cite{CGP12}, but it is also true due to Bosch, L\"{u}tkebohmert and Raynauld \cite[$\S$6.2 Proposition D.4(b)]{BLR12} since all our group schemes are affine. This proposition and its proof generalizes in a straightforward manner to the case where $x$ is replaced by finite set  $D$ of closed points of $C$.


Given an effective reduced divisor $D\subset C$ and a map $\sF$ which assigns $x\in D$ a facet $F(x)$ of $\Al$, which gives a triple $(G_{U}, \{\sG_{\sF(x)}\}_{x\in D}, Id)$. Here $\sG_{\sF (x)}$ is the group scheme over $\Spec\sO_{x}$ as in Theorem \ref{BT existence}. Combining Theorem \ref{BT existence} and Proposition \ref{grp over curve}, we can conclude this section with the following theorem:

\begin{theorem}\label{Bruhat-Tits Group Schemes over Curves}
	There exists a smooth affine group scheme $\sG_{\sF}$ over $C$, unique up to isomorphism, such that $\sG_{\sF }|_{U}$ is isomorphic to the constant group scheme $G$ over $U$, 
	and for every $x\in D$, the group  $\sG_{\sF}(\sO_x)$  is $\sO_x$-isomorphic to the parahoric subgroup $\hat{P}_{F(x)}$ compatible with the identity map between pull back of $G\times U$ and $\sG_{\sF(x)}$ over $\Spec\sK_x$ for $x\in D$.    		
\end{theorem}	

We call $\sG_{\sF}$ a \emph{Bruhat-Tits group scheme} attached to the pair $(D,\sF$). We might refer to a map $\sF$ as above as a set of \emph{parahoric data}. The group scheme $\sG_\sF$ will play an important role in our discussion of parahoric Hitchin systems.

\subsection{Parahoric Bundles and Higgs Bundles}\label{subsect:PB}

We keep our earlier notation:  $X^{*}(T)$ resp.\  $X_{*}(T)$ denotes the character group resp.\  cocharacter group of $T$,  $\sR(G,T)\subset X^{*}(T)$ the root system,  $\Delta\subset\sR$ the set of simple roots corresponding to our choice of $B$ and $\Al\subset X_*(T)\otimes\BR$ the fundamental alcove.

\begin{definition}
	Let $\sG_\sF$ be the Bruhat-Tits group scheme as in Theorem \ref{Bruhat-Tits Group Schemes over Curves}. Given a $k$-scheme $S$, a $\sG_\sF$-torsors over $S\times C$ is a scheme $\sE\rightarrow S\times C$ endowed with action of the group scheme $p^{*}\sG_{\sF}$ where $p:S\times C\rightarrow C$ such that there exists a $\fppf$ cover $\phi:Y\rightarrow S\times C$, along with a $\phi^{*}p^{*}\sG_{\sF}$ equivariant isomorphism between $\phi^{*}\sE$ and the trivial torsor $\phi^{*}p^{*}\sG_{\sF}$.
\end{definition}    
The folloing proposition follows directly from Heinloth's Proposition 1 in \cite{Hein10}, where the $\sG_{\sF}$ can be replaced by any smooth affine group scheme over $C$.	
\begin{proposition}
	The moduli functor $\mathbf{F}_{\sG_\sF}$:	
	\begin{equation*}
	\mathbf{F}_{\sG}: \mathsf{Sch}/k\rightarrow \mathsf{Groupoid},\;
	Y \mapsto \{\text{$\sG$-torsors over $C\times Y$}\}
	\end{equation*}
	 is represented by a smooth algebraic stack of locally finite type which is denoted by $\Bun_{\sG_\sF}$.
\end{proposition}

Now we want to look at the cotangent space at a point of $\Bun_{\sG_{\sF}}$ defined by a $\sG_{\sF}$ torsor.

\begin{lemma}
	The sheaf $\sH om_{\sO_C}(\Lie(\sG_{\sF}),\omega_C)$ is a locally free $\sO_X$-module and for 
	every  $\sG_{\sF}$-torsor $E$ over $C$, we have a natural identification of  the cotangent space
	$T^{*}_{E}\Bun_{\sG_{\sF }}$ with $H^0(C, E(\sH om_{\sO_C}(\Lie(\sG_{\sF}),\omega_C)))$.
\end{lemma}
\begin{proof}
	Because $\sG_{\sF}$ is a smooth affine group scheme over $C$, $\Lie(\sG_{\sF})$ is a locally free $\sO_{C}$ sheaf. Then $\sH om_{\sO_C}(\Lie(\sG_{\sF}),\omega_C)$ is also locally free. Deformation theory tells us that the tangent space of $\Bun_{\sG_{\sF }}$ at $E$ may be identified with
	$H^{1}(X,E(\Lie(\sG_{\sF})))$. So by Serre duality, the cotangent space at $E$ is identified with
	$H^{0}(X, E(\sH om_{\sO_C}(\Lie(\sG_{\sF}),\omega_C)))$ as asserted.
\end{proof}

In the following, we will give a better description of $T^{*}\Bun_{\sG_{\sF}}$.
We look at sections of $\sH om_{\sO_C}(\Lie(\sG_{\sF}),\omega_C)$ over $\BD_{x}=\Spec \sO_{x}$, the formal disc around $x\in D$. Using Lemma \ref{lemma:structure of parahoric}, we know sections of $\Lie(\sG_{\sF})$ over $\BD_{x}$ can be written as \begin{equation}
\gt(\sO)\oplus(\oplus_{\alpha\in \sR(G,T)}\g_{\alpha}(m^{-\lfloor\min\{\alpha|\sF(x)\}\rfloor}))
\end{equation}
We know that $\gt$ and the $\{\g_\alpha\oplus\g_{-\alpha}\}_{\alpha\in \sR_+(G,T)}$ are mutually perpendicular for the Killing form, with each $\g_\alpha$ isotropic, but $\g_\alpha\times\g_{-\alpha}\to k$ perfect. Since $\sF(x)\subset \Al$, $\lfloor\min\{\alpha|\sF(x)\}\rfloor=0,1$ or $-1$ for $\alpha\in\sR(G,T)$, we can check that the Killing form induces a perfect pairing for all $\alpha\in\sR(G,T)$:
\begin{equation}\label{pairing}
\g_{-\alpha}(m^{-\lfloor\min\{-\alpha|\sF(x)\}\rfloor})\times \g_{\alpha}(m^{\lfloor1-\min\{\alpha|\sF(x)\}\rfloor})\rightarrow m
\end{equation}
To sum up:
\begin{proposition}\label{parahoric dual}
	We denote the sheaf $E(\sH om_{\sO_C}(\Lie(\sG_{\sF}),\omega_C))$ by $E(\sF)\otimes\omega_{C}(D)$. Then sections of $E(\sF)$ over $\BD_{x}$ for $x\in D$ is 
	\begin{equation}
	\gt(m)\oplus(\oplus_{\alpha\in \sR(G,T)}\g_{\alpha}(m^{\lfloor1-\min\{\alpha|\sF(x)\}\rfloor}))
	\end{equation}
	here we use $\omega_{C}(D)$ is due to that the pairing \eqref{pairing} takes value in $m$.
\end{proposition}    

%

Now we can introduce the definition of parahoric Higgs bundles over $C$, and prove that the moduli functor of parahoric Higgs bundles is represented by an algebraic stack.
\begin{definition}
	A \emph{ parahoric} Higgs bundle over $C$ is a pair $(E,\theta)$, where $E$ is $\sG_{\sF}$ torsor, and $\theta\in H^{0}(X,E(\sF)\otimes\omega_{C}(D))$ here $E(\sF)$ is defined as in Proposition \ref{parahoric dual}. 
\end{definition}

The following theorem is well known, but we give a proof here for the completeness.
\begin{theorem}\label{rep of para hig}
	The moduli functor $\mathbf{H}_{\sG_\sF}: \mathsf{Sch}/k\to \mathsf{Groupoid}$ which assigns to a $k$-scheme $Y$  the groupoid of 
	pairs $(\sE,\Theta)$, where 
	$\sE$ is a 	$\sG_\sF$-torsor over $Y$ and $\Theta$ is a Higgs field for $\sE$ over $Y\times C$, is represented an algebraic stack (denoted by $\Hig_{\sG_{\sF}}$).
\end{theorem}
\begin{proof}
	We only need to verify the forgetful map $\mathbf{H}_{\sG_{\sF }}\rightarrow \Bun_{\sG_\sF}$ 
	is represented by a scheme. This follows from the following:

	\textit{
		Let $Y\rightarrow Z$ be a flat morphism of finite presentation and $\mathscr{F}$ a quasi-coherent $\sO_{Y}$-module  of finite presentation
		which is flat over $Z$.  Then there exists a quasi-coherent sheaf $W$ over $Z$ of finite presentation such that the $Z$-scheme $\Spec (Sym^{\bullet}_{\sO_Z}W)$ represents the functor $\mathsf{Sch}/Z\to \mathsf{Ab}$ which assigns  to every $Z$-scheme $Z'$, the abelian group $\Gamma(C\times_{Z}Z',f^*\mathscr{F})$.
}	
	
	For a proof, see \cite[Lemma 7.15]{CSW17}. We apply this to the following situation: let be given a morphism $Z\rightarrow \Bun_{\sG_{\sF}}$ and denote the corresponding principal bundle over $Y:=Z\times C$ by $\sE$ and  take $\mathscr{F}:=\sE(\sF)\otimes pr_C^{*}(\omega_{C}(D))$. Then we know that 
	\begin{equation*}
	\mathbf{H}_{\sG_{\sF}}\rightarrow \Bun_{\sG_{\sF}}
	\end{equation*}
	is represented by scheme, and since $\Bun_{\sG_{\sF}}$ is a smooth algberaic stack of finite type, it follows that $\mathbf{H}_{\sG_{\sF}}$ is represented by an algebraic stack $\Hig_{\sG_{\sF}}$.
\end{proof}
As $E(\sF)(\sO_{x})$ is a subgroup of $E(\g)(\sO_{x})$ for $x\in D$, then the Chevalley map $\chi$ induces a map $E(\sF)(\sO_{x})\rightarrow\gc(\sO_{x})$ for $x\in D$.
\begin{definition}[Hitchin map, Hitchin base]
	The \emph{Hitchin map} is the natural map
	\begin{equation*}
	h_{\sF}: \Hig_{\sG_{\sF}}\to H^{0}(C,\gc\times_{\mathbb{G}_m}\omega_{C}(D)),\;\;
	(E,\theta)\mapsto \chi(\theta)
	\end{equation*}
	where $\chi$ is the Chevelley map and and the \emph{Hitchin base} is the closure of its image (denoted by $\sA_{\sF}$).
\end{definition}   	
As we can see in \cite{BK18} and \cite{SWW19}, the map $h_{\sF}$ is far from being dominant, so that $\sA_{\sF}$ will be a proper subscheme of $H^{0}(C,\gc\otimes\omega_{C}(D))$. 

\subsection{Parabolic Case}
In this subsection, to get a better intuition of the Proposition \ref{parahoric dual}, we will analyze parabolic case,  where we can tranlate $\sG_{\sF}$ Higgs bundles into parabolic Higgs bundles. We make the following assumption: 
\begin{assumption}\label{para assm}
	The map $\sF: x\in D\mapsto F(x)$ takes values in the faces of $\Al$ that contain the origin.
\end{assumption}

Recall the Example \ref{example:parabolic}, these are the  faces of $\Al$ that do not lie in the hyperplane $\alpha_{max}=1$, or equivalently, for which the corresponding subset of $\hat{S}$ is contained in $\hat{S}\ssm\{s_0\}$.  Every such face corresponds to a standard parabolic of $G$ and so we may now regard $\sF$ as a map $\sP$  which assigns to every $x\in D$ a standard parabolic $\sP(x)\subset G$  defined by $\sF(x)$. For this reason, we sometimes write $\sG_\sP$ for $\sG_\sF$.

We define the moduli functor 
\begin{equation}
\mathbf{F}_{G,\sP}: \mathsf{Sch}/k\to \mathsf{Groupoid}
\end{equation}
which assigns to a $k$-scheme $Y$ the groupoid of principal $G$-bundles $\sE$ over $Y\times C$ with for every $x\in D$ a section of $\sE(G/\sP(x))$ given over $Y\times\{x\}$.

We denote the moduli stack of $G$-bundles with parabolic structure $P$ at $x$ by $\Bun_{G,\lambda}$. We choose a root space decomposition with respect to the maximal torus $T$. $\g=\gt\oplus_{\alpha\in R(G,T)}\g_{\alpha}$
here $R(G,T)$ is the set of roots. We denote the complete local ring of $X$ at $x$ by $\sO$, we may choose a uniformizer $z$ at $x$, and fix an isomorphism $\sO\simeq k[[t]]$. We denote its fraction field by $\sK$ isomorphic to $k((t))$.
    
In the remaining part of this subsection, we make the following assumption: 
    \begin{assumption}\label{para assm}
    	$\langle\alpha_{max},\lambda\rangle<1$ and $\lambda\in X_{*}(T)\otimes\mathbb{Q}$.
    \end{assumption}
Under this assumption, $\sG_{\lambda}(\sO_{x})$ is the standard parahoric subgroup corresponds to $P_{\lambda}$. Then it is known that:
\begin{proposition}\cite{Hein10}\label{identification}
	The moduli functor $\mathbf{F}_{G,\sP}$ is represented by a smooth algebraic stack, which we denote by $\Bun_{G, \sP}$  
	(called the \emph{moduli stack of quasi-parabolic $G$-bundles}) and we have
	$\Bun_{G,\sP}\simeq\Bun_{\sG_\sP}$.
\end{proposition} 
Let us now explain this identification. As we assume for each $x\in D$ the assigned facet $F(x)$ is not contained in the hyperplane $\alpha_{\max}=1$, we have a morphism  $\sG_{\sP}\rightarrow G\times C$ with image $\hat{P}_{F(x)}$ in $G(\sO_{x})$ which is the preimage of $P(x)$ under the evaluation map $ev:\hat{G}\rightarrow G\subset P$. Then a $\sG_{\sP}$ torsor over $X$ amounts to say a $G$ torsor along with inductions $s_{x}\in G/\sP(x)$ at each $x\in D$.

\begin{remark}
	$\Bun_{G,\sP}$ is called  \emph{moduli stack of flagged $G$-bundles} in \cite{HS10}. In what follows, we will not distinguish between $\Bun_{G,\sP}$ and $\Bun_{\sG_\sP}$.  
\end{remark}

In parabolic cases, the Higgs field has a more explicit description. Since $\sF(x)$ contains the origin, and we denote $P(x)$ the corresponding parabolic subgroup of $G$. Let $\hat \gp$ resp.\ $\hat{\gn}_{\gp}$ be the preimage of $\gp(x)$ resp.\ $\gn_{\gp(x)}$ (the nilpotent radical) under the reduction map $\g\otimes_k\sO\to \g$. Then we can interpretation the pairing \eqref{pairing} as follows:
\begin{corollary}\label{dual}    	   	
	The $\sO$-linear extension of the Killing form restricts to a pairing $\hat{\gp}\times \hat{\gn}_{\gp}\rightarrow m$
	that is perfect over $\sO$ (recall that $m$ is the maximal ideal of $\sO$). 
\end{corollary}

Sections of the sheaf $E(\sH om_{\sO_C}(\Lie(\sG_{\sF}),\omega_C)$ over the formal disc $\BD_{x}=\Spec\sO_{x}$ can be described as $\hat{\gn}_{\gp}dt/t$ where $t\in\sO$ is a uniformizer (so that $dt/t$ generator of $\omega_{C,x}$) and $\gp$ is the Lie algebra of the parabolic subgroup $P(x)$. The residue map defines a surjection  $E(\g)\otimes\omega_{C}(D)\rightarrow \bigoplus_{x\in D} i_{x,*}\g/\gn_{\gp(x)}$ of $\sO_C$-modules. Let us denote its kernel by 
$E(\sF)\otimes\omega_{C}(D)$ so that we have an exact sequence
\begin{equation}\label{para higgs field}
0\rightarrow E(\sF)\otimes\omega_{C}(D)\rightarrow E(\g)\otimes\omega_{C}(D)\rightarrow \bigoplus_{x\in D} i_{x,*}\g/\gn_{\gp(x)}\rightarrow 0.
\end{equation}   
By Proposition \ref{parahoric dual} we may identify  $E(\sF)$ with $E(\sH om_{\sO_C}(\Lie(\sG_{\sF}),\omega_C))$.  Roughly speaking, the value of the Higgs field $\theta$ at $x\in D$ lies in $\gn_{\gp(x)}$, in particular it is nilpotent at $x\in D$. The order one pole at $x$ is explained by the duality in Corollary \ref{dual}.

\section{Geometry of Generic Fibers}\label{Section: para gen fibers}

In this section, we will describe generic fibers of parahoric Hitchin maps. We should mention that Baraglia, Kamgarpour, and Varma showed that over $\mathbb{C}$, the parahoric Hitchin map gives rise to an algebraically completely integrable system which generalizes  the nonparabolic case introduced in Hitchin's seminal paper \cite{Hit87S}. Their proof is complex-analytic in the  sense that they show that functions induced by Hitchin maps are Poisson commuting, and that in order to prove the properness of the Hitchin map, they make use of the Simpson correspondence over punctured Riemann surfaces. These methods cannot be applied in our setting. However their use of the aforementioned Simpson correspondence, and the work of Balaji and Seshadri in \cite{BS14}, motivates us to study the moduli stack of Higgs bundles over a smooth root stack $\mathscr{C}$ with coarse moduli space $C$. This will lead  us to set up an equivalence between the category of Higgs bundles over a root stack and the category of parahoric Higgs bundles over the underlying curve.

\subsection{Conjugacy classes of cyclic subgroups of $G$}\label{subsect:cyclic}
Let $r$ be a positive integer relative prime to the characteristic $p$ of $k$. Then the subgroup $\mu_r\subset k^\times=\mathbb{G}_m(k)$ of $r$th roots of unity consists of $r$ elements. Our first goal is to describe the $G$-conjugacy classes in $\Hom(\mu_r,G)$. 

A first observation is that any $\rho\in \Hom(\mu_r,G)$ lands in a maximal torus of $G$. Since this torus is conjugate to $T$, a $G$-conjugacy class of $\Hom(\mu_r, G)$ meets $\Hom(\mu_r, T)$ in a $W$-conjugacy class.  We next set up  $W$-equivariant group isomorphism between $\Hom(\mu_r, T)$ and $X_*(T)\otimes r^{-1}\BZ/\BZ$.  Since $\mu_r$ is a cyclic group, any $\rho\in \Hom(\mu_r, T)$ takes its values in a one-parameter subgroup of $T$.
So there is a primitive $\lambda_o\in X_*(T)$ (primitive means: $\lambda_o$ is injective) such that $\mu$ factors through
$\lambda_o$. In that case there exists an integer $a$ such that $\rho(\zeta)=\lambda_o^a(\zeta)$. It is now straightforward to check
that  

\begin{lemma}\label{lemma:}
	The element $\lambda_o\otimes \frac{a}{r}\in X_*(T)\otimes r^{-1}\BZ/\BZ$  only depends on $\rho$ and the resulting map
	$\Hom(\mu_r, T) \to  X_*(T)\otimes r^{-1}\BZ/\BZ$
	is a $W$-equivariant isomorphism of abelian groups. 
\end{lemma}

Thus the $G$-conjugacy class of $\rho$ determines a $W$-orbit in $X_*(T)\otimes r^{-1}\BZ/\BZ$ and vice versa.
Since $\hat W$ is a semi-direct product of $W$ and 
$X_*(T)$, this corresponds to the $\hat W$-orbit of $\lambda_o\otimes \frac{a}{r}$. 
As  the fundamental alcove $\Al$ is a strict fundamental domain for the $\hat W$-action on $X_*(T)\otimes \BR$, we  therefore represent this  
$\hat W$-orbit by a unique element $\lambda$ of $\Al\cap X_*(T)\otimes r^{-1}\BZ$. We sum up:

\begin{corollary}\label{cor:alcoverep}
	Every $G$-conjugacy class in $\Hom(\mu_r, G)$ is uniquely represented by a rational  cocharacter $\lambda=\lambda_o\otimes \frac{a}{r}\in \Al\cap (X_*(T)\otimes r^{-1}\BZ)$ (the representative being $\lambda_o^{a}|\mu_r$) and this defines a bijection between the $G$-conjugacy classes in $\Hom(\mu_r, T)$ and 
	$\Al\cap (X_*(T)\otimes r^{-1}\BZ)$.
\end{corollary}
%

\subsection{Equivariant Higgs Bundles over a Formal Disc}\label{subsection:quotient disc}
In this subsection, we will describe Higgs bundles over a formal disc endowed with a finite cyclic group action.  Let $\sO, \sK$ be as before. Let $\sK'/\sK$ be a field extension of degree $r$, with $r$ be as above, so 
relatively prime to $p$.  This  is then a separable field extension. The integral closure of $\sO$ in $\sK'$, denoted 
$\sO'$,  is also a  DVR and if $t$ the uniformizer of $\sO$ and 
$z\in \sK'$ satisfies $z^r=t$, then $\sO'=k[[z]]$. For given $t$, such a $z$ is unique up to multiplication by an element of  $\mu_r\subset k^\times=\mathbb{G}_m(k)$ and indeed, 
the Galois group $\Gal(\sK'/\sK)$ may be identified with $\mu_r$. We extend the valuation on $\sK$ to a $r^{-1}\BZ$-valued one on $\sK'$ so that $z$ has valuation $r^{-1}$. 

Let us abbreviate  $\BD_{x}:=\Spec\sO$, $\BD^{\times}_{x}:=\Spec\sK$, $\BD_{x'}:=\Spec\sO'$ and $\BD_{x'}^\times:=\Spec\sK'$ and denote by $\pi: \BD_{x'}\to \BD_x$ the projection. Then $\pi$ is a ramified cover with ramification index $(r-1)$.
The natural $\mu_r$-action on tangent space of $\BD_{x'}$ at $x$ is given by the inclusion $\mu_r\subset\mathbb{G}_m(k)$. 

Since $G$ is semisimple and simply connected, a $G$-bundle over a smooth affine curve is trivial. In characteristic 0, Kumar, Narasimhan and Ramananthan \cite[Proposition 1.3]{KNR97} attribute this to Harder \cite[Satz 3.3]{Har67}. In the general case, this is due to Drinfeld and Simpson \cite[Theorem 3]{DS95}. For what follows, it is worth noting that a $G$-bundle over $\BD_{x'}$ is necessarily trivial.

\begin{definition}\label{def:orbifoldbundle}\label{def:orbifold higgs}
	A \emph{$G$-bundle over $\BD_{x'}/\mu_r$} is a $G$-bundle $E$ over $\BD_{x'}$
	endowed with an $\mu_r$-action , such that this action for some trivialization over $\BD_{x'}$ takes the form:
	\begin{equation}
	(z,g)\in \BD_{x'}\times G\mapsto (\zeta z,\rho(\zeta)g)\in \BD_{x'}\times G,
	\end{equation}
	for some $\rho\in \Hom(\mu_r,G)$. 
	
	A \emph{$G$-Higgs bundle} is pair $(E,\Theta)$, where $E$ is a $G$-bundle over $\BD_{x'}/\mu_r$ and $\Theta$ is  a $\mu_r$-equivariant section of $E(\g)\otimes\omega_{\BD_{x'}}$.	
\end{definition} 

In fact, for a $\rho\in \Hom(\mu_r,G)$, there is underlying group scheme $\sG^\rho$ over $\sO$, which is a `twisted form' of $G\times \BD_{x}$ in the sense that  $\sG^\rho$ after the base change 
$\sO\subset \sO'$ is isomorphic to $G\times\BD_{x'}$, but with the Galois action given as follows: $\zeta\in \mu_r$ takes $g(z)\in G(\sO')$ to $\rho(\zeta)g(\zeta^{-1}z)\rho(\zeta)^{-1}$, that is to say $\sG^{\rho}$ is an inner form of the costant group scheme $G$ over $\sO$. Then a $G$-bundle over $\BD_{x'}/\mu_r$ is the same thing as a $\sG^{\rho}$-torsor over $\BD_{x}$. In any case, it is clear that
the group $\sG^{\rho}(\sO)$ of its $\sO$-points,  which is the group of Galois-invariant elements of $\sG^{\rho}(\sO')=G(\sO')$, is precisely the group $\Aut^{\mathsf{C}}(E)$ defined below, here $\mathsf{C}$ is the conjugacy class of $\rho$. 

\begin{remark}\label{rem:}
	Note that for the trivialized  $E$ as above, any other trivialization  differs from the given one by left multiplication with some  $u\in G(\sO')$.
	This then changes the action of $\mu_r$ into the one which assigns to $\zeta\in\mu_r$ by left multiplication with $u\rho (\zeta) u^{-1}$.  If we want the latter
	to be an element of $G(k)$, then this means that we have  merely changed $\rho$ inside its $G$-conjugacy class. So the conjugacy class $\mathsf{C}$ of
	$\rho$ is a complete invariant of the isomorphism type of the $G$-bundle $E$; we then say that \emph{$E$ is of type $\mathsf{C}$}. 
	Corollary \ref{cor:alcoverep} provides us then with a natural  choice: we can and will assume that $\rho$ is given by a rational  cocharacter $\lambda=\lambda_o\otimes \frac{a}{r}\in \Al\cap (X_*(T)\otimes r^{-1}\BZ)$. 
\end{remark}

Observe that the restriction of $E$ over $\BD_{x'}/\mu_r$  to $\BD^\times_y$ is isomorphic to the pull-back of a 
$G$-bundle over $\BD_x^\times$ \emph{with the ensueing $\mu_r$-action}. If $\rho$ is trivial, then this is of course true for all of $E$.
\\

We denote the automorphism group of an $G$-bundle $E$ over $\BD_{x'}/\mu_r$ of type $\mathsf{C}$ by $\Aut^{\mathsf{C}}(E)$. Any $\Phi\in \Aut^{\mathsf{C}}(E)$ is in terms of the above trivialization given by an element $\phi\in G(\sO')=\Mor_k(\BD_{x'},G)$  which sends $(z,g)$ to $(z, \phi(z)g)$ and  commutes with the above action. This means that 
\begin{equation}
\phi (\zeta z)=\rho (\zeta)\phi(z)\rho(\zeta)^{-1}.
\end{equation}  
Our goal is to identify $\Aut^{\mathsf{C}}(E)$ with a parahoric subgroup of $G(\sO)$. 

The map $z\mapsto \lambda_o^a(z)$ defines an element $\eta\in G(\sK')$ with the property that $\eta(\zeta z)=\rho(\zeta)\eta(z)=\eta(z)\rho(\zeta)$.
So if we put $\phi^\eta:=\eta^{-1}\phi\eta$, then $\phi$ defines an automorphism of type $\rho$ if and only if $\phi^\eta(\zeta z)=\phi^\eta(z)$, in other words, 
if and only if $\phi^\eta\in G(\sK)$. 

\begin{proposition}\label{prop:bundleconversion}
	Assume that $E$ is a $G$-bundle over $\BD_{x'}/\mu_r$ and let  $F$ the associated facet of the fundamental alcove. Then the map $\Phi\mapsto \phi^\eta$ identifies $\Aut^{\mathsf{C}}(E)$ with the parahoric $\hat P_F$.
\end{proposition}
\begin{proof}
	In view of what we already established, it remains to show that $\hat P_F=G(\sK)\cap \eta G(\sO')\eta^{-1}$. We first prove
	that the right hand side contains $\hat P_F$. 
	
	It is clear that $T(\sO)\subset G(\sK)\cap \eta G(\sO')\eta^{-1}$. Let $\alpha\in \sR(G, T)$ and denote by $U_\alpha :\mathbb{G}_a\to G$ the
	associated root group. Then for any nonzero $f\in \sK$, 
	\begin{equation}
	\eta U_\alpha (f)\eta^{-1}=U_\alpha(z^{a\langle\alpha,\lambda_o \rangle}f)
	\end{equation}
	The right hand side lies in $G(\sO')$ if and only if the valuation of $f$, $v(f)$, is $\ge -\langle\alpha,\lambda \rangle$.
	This is equivalent to: $v(f)\ge -\min\{\alpha|F\}$. This proves that 
	$\hat P_F\subset G(\sK)\cap \eta G(\sO')\eta^{-1}$ and that this inclusion is no longer true if we replace $F$ by proper facet $F'$ of $F$. Since $G(\sK)\cap \eta G(\sO')\eta^{-1}$ contains $\hat{P}_{F}$, it is a parahoric subgroup. The theory of Tits systems then implies that we must have $G(\sK)\cap \eta G(\sO')\eta^{-1}=\hat P_F$.
\end{proof}

We also have an interpretation for the equivariant Higgs bundles:

\begin{proposition}\label{prop:higgsconversion}
	The Lie algebra version of Proposition \ref{prop:bundleconversion} identifies a Higgs field supported by $E$ with 
	a Higgs field on the root stack associated to the parahoric group $\sG_\sF$. To be precise, conjugation with $\eta^{-1}$ defines an isomorphism  
	\begin{equation}
	H^{0}(\BD_{x'},E(\g)\otimes\omega_{\BD_{x'}})^{\mu_r}\cong H^{0}(\BD_{x},E(\sF)\otimes\omega_{\BD_{x}}(x)),
	\end{equation}
	where $E(\sF)$ is defined by the exact sequence \ref{para higgs field}.
\end{proposition}
\begin{proof}
	The proof is similar. Any  $\Theta\in H^{0}(\BD_{x'},\g\otimes\omega_{\BD_{x'}})^{\mu_r}$ can be written as $\theta dz/z$ with 
	$\theta\in \g(m')$. Since $\Theta$ is $\mu_r$-invariant, we have $ \gamma\theta=\rho(\gamma)^{-1}\theta\rho(\gamma)$.
	Hence  $\theta^\eta:=\eta^{-1}\theta\eta$ is an element of $\g(\sK)$. If we write this element out according to the root space decomposition:
	\begin{equation*}
	\textstyle \theta^\eta=\theta^\eta_{\gt} +\sum_{\alpha\in \sR(G,T)}u_\alpha(\theta^\eta_{\alpha}),
	\end{equation*}
	where $\theta^\eta_{\gt}\in \gt(\sK)$  and $\theta^\eta_{\alpha}\in \sK$,  then
	\begin{equation*}
	\textstyle  \theta=\theta_{\gt}+ \sum_{\alpha\in \sR(G,T)} t^{\langle\alpha,\lambda\rangle}u_\alpha(\theta^\eta_{\alpha}).
	\end{equation*}
	For  $\theta$ to have positive valuation, it is equivalent  that $\theta^\eta_{\gt}$ has positive valuation, i.e., $\theta^\eta\in \gt(m)$, and that for all roots $\alpha$, we have 
	$v(\theta^\eta_{\alpha})> -\langle\alpha,\lambda\rangle$. This means precisely that $\theta^\eta\in E(\sF)(\BD_{x})$. This proves that  the conjugation by $\eta$ gives the stated isomorphism.
\end{proof}

We  sum up the preceding as the  following theorem. It will be crucial for our description of generic fibers of parabolic Hitchin systems.

\begin{theorem}\label{proposition:local equivalence of higgs}
	Let $\mathsf{C}$ be a $G$-conjugacy class in $\Hom(\mu_r, G)$ and let  $F$ be the associated 
	facet of the fundamental alcove. Then the above construction defines  an equivalence of  categories
	\begin{equation}
	\{\text{$G$-Higgs bundles over $\BD_{x'}/\mu_r$ of type $C$} \}\xrightarrow{\sim} 
	\{\text{$\sG_{F}$-Higgs bundles over $\BD_{x}$}\}.
	\end{equation}
\end{theorem}

In the next subsection, we will obtain a global version of this correspondence. 
In order to discuss generic fibers of parahoric Hitchin maps, we shall define an analogue of the commutative group scheme $\sJ_{a}$ over the formal disc $\BD_{x'}$ as follows:
\begin{definition}\label{defn:centralizer over disc stack}
	Given $a\in H^{0}(\BD_{x'},\gc\otimes\omega_{\BD_{x'}}^{\times})^{\mu_r}$, we define $\sJ_{a,\BD_{x'}}:=a^{*}(\sJ\times_{\mathbb{G}_{m}}\omega_{\BD_{x'}}^{\times})$, which is a smooth commutative group scheme over $\BD_{x'}$ and there is a $\mu_r$ action on $\sJ_{a,\BD_{x'}}$ induced by $\mu_r$ action on $\BD_{x'}$.
\end{definition}
By \cite[Lemma 2.4.2]{Ngo10}, there is an embedding $\sJ\rightarrow (\Res_{\gc}^{\gt}T)^{W}$, here $\Res_{\gc}^{\gt}T$ is the Weil restriction of $\gt\times T$ to $\gc$. It is not difficult to see that $\Ad_{\eta^{-1}}$ also acts on $\Res_{\gc}^{\gt}T$. Since the Weyl group action on $\Res_{\gc}^{\gt}T$ is induced from its action on $\gt$, so $\Ad_{\eta^{-1}}$ action is compatible with the Weyl group action. Then $Ad_{\eta^{-1}}\sJ_{a,\BD_{x'}}$ is a well defined group scheme over $\BD_{x'}$ isomorphic to $\sJ_{a,\BD_{x'}}$. 
Since  $r$ is coprime with $\chr(k)$, a result by Edixhoven \cite[Proposition 3.5]{Ed92} implies  that the group scheme $(\pi_{*}Ad_{\eta^{-1}}\sJ_{a,\BD_{x'}})^{\mu_r}$ is smooth over $\BD_{x}$. To sum up:

\begin{proposition}\label{prop:monodromic centralizer}
	The group scheme $(\pi_{*}Ad_{\eta^{-1}}\sJ_{a,\BD_{x'}})^{\mu_r}$ is well defined and is a smooth commutative group scheme over $\BD_{x}$.
\end{proposition}

Now we have made preparations for our description of generic fibers of parahoric Hitchin maps.
\subsection{Hitchin System over Root Stacks}

In this subsection, we define Higgs bundles over a root stack and relate them with parahoric Higgs bundles over the underlying coarse moduli space, which is our curve $C$.  For the definition of a general root stack we refer to the Appendix.

Let $(D,\sF)$ be a parahoric data as before, and suppose for each $x\in D$ given  a $\lambda(x)\in F(x)\cap (X_*(T)\otimes \BZ_{(p)})$.
We choose a positive integer $r$ not divisible by $p$ such that $r\lambda(x)\in X_{*}(T)$ for all $x\in D$. In view of Corollary \ref{cor:alcoverep} this amounts to giving  for  each $x\in D $ a conjugacy class in $\Hom(\mu_r, G)$ which we denote as $\mathsf{C}(x)$. Thus the parabolic data $(D,\sF)$ gives rise to a map $\mathsf{C}:D\rightarrow \{\text{Conjugacy class of maps} \Hom(\mu_r,G)\}$.

\begin{definition}\label{root stack}
	We associate with $(D,\sF,r)$ an $r$-th root stack $\mathscr{C}$ which is defined as the category fibered in groupoids over $C$ whose objects for the $C$-scheme $f:Y\to C$ are the triples $(\sL,\phi,s)$ where
	$\sL$ is a line bundle over $Y$, $\phi$ is an isomorphism of line bundles $\phi: \sL^{\otimes r}\simeq f^{*}\sO_{C}(D)$ and $s\in H^{0}(Y,\sL)$ is such that $\phi(s^{r})$ is the canonical section $1\in H^{0}(Y,f^{*}\sO_{X}(D))$. We denote the natural map $\pi:\mathscr{C}\rightarrow C$. 
\end{definition}
For more details about the following facts we refer the reader to \cite{Beh14} or the Appendix.
\begin{itemize}
	\item[(i)]$\mathscr{C}$ is proper and the natural map $p:\mathscr{C}\rightarrow C$ is the formation of its coarse moduli space. If $r$ is invertible in $k$, then $\ssC$ is a Deligne-Mumford stack; 
	\item[(ii)] $\mathscr{C}$ has a canonical line bundle $\omega_{\mathscr{C}}$  (sometimes simply denoted as $\omega$).
\end{itemize}

For each $x\in D$, we let  $\BD_{x'}$ be as in Subsection \ref{subsection:quotient disc};  so $\BD_{x'}\to \BD_x$ is a ramified covering with ramification index $r-1$. We put $U:=C\ssm D$, so that we have a natural morphism $p:U\sqcup(\sqcup_{x\in D} \BD_{x'})\rightarrow C$.
Note that the line bundle $(\sO_{U},\{\sO_{\BD_{x'}}(x')\}_{x\in D})$ over $U\sqcup(\sqcup_{x\in D} \BD_{x'})$ has the property that its $r$th tensor power is isomorphic to $ p^{*}\sO_{C}(D)$. This means that $p:U\sqcup(\sqcup_{x\in D} \BD_{x'})\rightarrow C$ factors through $\mathscr{C}$. Since $\mu_r$ acts freely on $\BD_{x'}^{\times}$ with  quotient naturally identified with $\BD_{x}^{\times}=U\cap\BD_{x}$ for each $x\in D$, we can glue $U$ with $\BD_{x'}/\mu_r$ along $\BD_{x}^{\times}$ for all $x\in D$. We conclude:
\begin{proposition}\label{covers of root stack}
	The root stack $\mathscr{C}$ is isomorphic to $U\sqcup_{\BD_{x}^{\times},x\in D}(\BD_{x'}/\mu_r)$.
\end{proposition}
Intuitively, $\mathscr{C}$ is decomposed into $U$ and for each $x\in D$ a copy of $B\mathbb{\mu}_{r}$, where $B\mathbb{\mu}_{r}$ is the classifying stack of $\mathbb{\mu}_{r}$.
\begin{remark}
	We have a cover $U\sqcup(\sqcup_{x\in D} \BD_{x'})$ of $\mathscr{C}$. It is not an \'etale cover because it is not of finite type. If we want an etale cover, then we need to replace formal discs $\{\BD_{x}\}_{x\in D}$ by an appropriate etale open neighbourhood of $D$ in $C$. But the use of a formal disk  will make our arguments  easier.
\end{remark}

%

Now we can describe $G$ bundles over $\mathscr{C}$ of type $\mathsf{C}$.
\begin{propdef}\label{G over root}
	Let $G$ be semi-simple and simply connected as before and let $\mathsf{C}$ assign to each $x\in D$ a $G$-conjugacy class $\mathsf{C}(x)$ in $\Hom(\mu_r, G)$.
	Then by uniformization theorem in \cite{DS95}, and Beauville and Laszlo's gluing lemma in \cite{BL94}, $G$ bundles over $U$ and $\mathbb{D}_{x'}$ are trivial. As a result, to give a principal $G$-bundle over $\mathscr{C}$ \emph{of type $\mathsf{C}$} amounts to giving a triple 
	\begin{equation}
	(E_{U}, \{E_{\BD_{x'}}\}_{x\in D},\{\phi_{x}\}_{x\in D})
	\end{equation}
	where $E_{U}$, resp. $E_{\BD_{x'}}$, is a $G$-bundle over $U$, resp. a $G$ bundle over $\BD_{x'}/\mu_r$ of type $\mathsf{C}(x)$, and 
	$\phi_{x}:E_{\BD^\times_{x'}}\rightarrow \pi^{*}E_{U}|_{\BD_{x}^{\times}}$ is a $\mu_r$-equivariant isomorphism.	
\end{propdef}
%

Let $\mathscr{C}$ be as before and denote its canonical line bundle by $\omega_{\mathscr{C}}$.
We also fix for every $x\in D$ a $G$-conjugacy class $\mathsf{C}(x)\subset \Hom(\mu_r, G)$ and denote the map $x\in D\mapsto \mathsf{C}(x)$ by $\mathsf{C}$. We then write $\lambda(x)\in (X_*(T)\otimes r^{-1}\BZ)\cap \Al$ for the corresponding rational character, $F(x)$ the facet the fundamental alcove which contains $\lambda(x)$ in its relative interior. If the conjugacy class $\mathsf{C}(x)$ is standard, then we write $P(x)$ for the associated parabolic. Recall that $\sF: x\in D\mapsto F(x)$ defines a set of parahoric data for $(C,D)$ which also give rises a map $\mathsf{C}:x\in D\mapsto \mathsf{C}(x)$.

\begin{definition}[Higgs bundle over a root stack]
	A \emph{$G$-Higgs bundle over  $\mathscr{C}$} of type $\mathsf{C}$ is a pair $(E,\theta)$, where $E$ is a $G$-bundle over $\mathscr{C}$ of type $\mathsf{C}$ and $\theta\in H^{0}(\mathscr{C},E(\g)\otimes\omega_{\mathscr{C}})$.
\end{definition}
We refer to  \cite{Beh14} for a good introduction to the  cohomology of a coherent sheaf on a root stack. The following theorem is a global version of Theorem \ref{proposition:local equivalence of higgs} whose proof is direct.

\begin{theorem}\label{Higgs equi}
	By the description in Proposition \ref{G over root}, the construction in Theorem \ref{proposition:local equivalence of higgs} can be applied to define an equivalence:
	\begin{equation}
	\{\text{$G$-Higgs bundles over $\mathscr{C}$ of type $\mathsf{C}$}\}\xrightarrow{\sim} \{\text{$\sG_{\sF}$ Higgs bundles over C}\}
	\end{equation}
\end{theorem}
%
Now we can give an interesting corollary:
\begin{corollary}\label{cor:image}
	The image of parahoric Hitchin map is indeed an affine space.
\end{corollary}
For a parahoric Hitchin system over a Riemann surface, with structure group a semisimple simply connected algebraic group, Baraglia, Kamgarpour and Varma in \cite[Theorem 10]{BKV18} proved that the image has half dimension of the moduli stack $\Hig_{\sG_{\sF}}$. Our result actually can show that the image is indeed an affine space. 
\begin{remark}
	We need to point out that the Theorem \ref{Higgs equi} does not hold, if $G$ is not simply-connected. For example, if $G=\SO_{2n}$, Baraglia and Kamgarpour \cite{BK18} proved that if the parabolic data at a marked point corresponds to a very even orbit, the image of the parabolic Hitchin map is not irreducible. Actually, if $G$ is not simply-connected, the moduli stack $\Hig_{\sG_{\sF}}$ is not connected.
\end{remark}
Given an affine $k$-scheme $Y=\Spec R$,  $(Y\times U, \{\BD_{x,R}\}_{x\in D})$ can be treated as a cover of $Y\times\mathscr{C}$, here $\BD_{x,R}\simeq \Spec R[[t]]$. And similarly we can define $G$-Higgs bundles over $Y\times\mathscr{C}$. By Theorem \ref{rep of para hig} and Theorem \ref{Higgs equi}, we have the following:
\begin{corollary}
	The moduli functor $\mathbf{H}_{\mathscr{C}}: \mathsf{Sch}/k\to \mathsf{Groupoid}$ which assigns to an affine $k$-scheme $Y$ the groupoid of pairs
	$(\sE,\Theta)$, where $\sE$ is a $G$-bundle over $Y\times \mathscr{C}$  and $\Theta$ is an associated Higgs field, is represented by an algebraic stack $\Hig_{\mathscr{C}}$.
\end{corollary}

It is then clear that we now can define a  Hitchin map over the root stack $\mathscr{C}$  as a map:
\begin{equation*}
\Hig_{\mathscr{C}}\rightarrow \sH_{\sF}:=H^{0}(\mathscr{C},\omega^{\times}_{\mathscr{C}}\times_{\mathbb{G}_m}\gc),
\end{equation*}
Let us spell out that an $a\in H^{0}(\mathscr{C},\omega^{\times}_{\mathscr{C}}\times_{\mathbb{G}_m}\gc)$ is given by a triple 
\begin{equation}
(a(U),\{a(\BD_{x'})\}_{x\in D},\{\phi_{x}\}_{x\in D})
\end{equation} 
with $a(U)$ is section of $(\omega^{\times}_{X}|_{U})\times_{\mathbb{G}_m}\gc$, $a(\BD_{x'})$ an $\mu_r$ invariant section of $\omega^{\times}_{\BD_{x'}}\times_{\mathbb{G}_m}\gc$ and $\phi_{x}:\pi_{x}^{*}(a(U)|_{\BD_{x}^{\times}})\rightarrow a(\BD_{x'})|_{\BD_{x'}^{\times}}$ an $\mu_r$ equivariant isomorphism (here $\pi_{x}:\BD_{x'}\rightarrow\BD_{x}$ is the projection).
\begin{definition}
	For  every closed point $a\in\sH_{\sF}(k)$, we have a corresponding cameral cover:
	\begin{equation*}
	\tilde{\mathscr{C}}_{a}:=a^{*}(\gt\times_{\mathbb{G}_m}\omega^{\times}_{\mathscr{C}})
	\end{equation*}
\end{definition}
Then we spell out the cameral cover $\tilde{\mathscr{C}}_{a}$ as 
$(\tilde{U}_{a},\{\tilde{\BD}_{x',a}\}_{x\in D})$.
Notice that $\tilde{\BD}_{x',a}$ is $W$-cover of $\BD_{x'}$, which is also endowed with $\mu_r$ action induced from the $\mu_r$-action on $\BD_{x'}$.

The smooth affine commutative group scheme that  was introduced in the nonparabolic case in Theorem \ref{Univ cen} has now a natural generalization 
in this setting.

\begin{definition}
	We define a smooth commutative group scheme over $\mathscr{C}$:
	\begin{equation}
	\sJ_{a}:=a^{*}(\sJ\times_{\mathbb{G}_{m}}\omega_{\mathscr{C}}^{\times})
	\end{equation}
	here $\sJ\rightarrow\gc$ is the universal centralizer described in Theorem \ref{Univ cen}. 
\end{definition}
Here the pull back functor $a^{*}$ can be treated as a fiber product which is well defined for Deligne-Mumford stacks. To make it more explicit, as usual, we describe the group scheme as a system: 
\begin{equation*}
(a(U)^{*}(\sJ\times_{\mathbb{G}_{m}}\omega_{U}^{\times}),\{a(\BD_{x'})^{*}(\sJ\times_{\mathbb{G}_{m}}\omega_{\BD_{x'}}^{\times}),\phi_x^{*}\}_{x\in D}),
\end{equation*}
We  abbreviate  
\begin{equation*}
\sJ_{a,U}:=a(U)^{*}(\sJ\times_{\mathbb{G}_{m}}\omega_{U}^{\times}), \quad \sJ_{a,\BD_{x'}}:=a(\BD_{x'})^{*}(\sJ\times_{\mathbb{G}_{m}}\omega_{\BD_{x'}}^{\times}),
\end{equation*}
where we note that the latter comes endowed with a $\mu_r$-action.  Applying Proposition \ref{grp over curve} and \ref{prop:monodromic centralizer}, we thus find: 

\begin{corollary}\label{struc grp over curve}
	For a point $a\in H^{0}(\mathscr{C},\omega^{\times}_{\mathscr{C}}\times_{\mathbb{G}_{m}}\gc)$, we obtain a  smooth commutative group scheme 
	$\sJ_{a}^{\sF}$ over $C$ by gluing $\sJ_{a,U}$ and $(\pi_{*}Ad_{\eta^{-1}}\sJ_{a,\BD_{x'}})^{\mu_r})_{x\in D}$ via a natural group isomorphism
	$\sJ_{a,\BD_{x}}\simeq (\pi_{*}Ad_{\eta^{-1}}\sJ_{a,\BD_{x'}})^{\mu_r})|_{\BD_{x}^\times}$.
\end{corollary}

A torsor over $\sJ_{a}$ is always obtained by gluing a torsor over $U$ of $\sJ_{a,U}$ and a $\sJ_{a,\BD_{x'}}$-torsor over $\BD_{x'}/\mu_r$. Let $\Pic\sJ_{a}$ be the stack of $\sJ_{a}$-torsors over $\mathscr{C}$.
%
We now borrow the following lemma which is due to Balaji and Seshadri \cite[Theorem 4.1.6]{BS14}.
\begin{lemma}\label{equiva torsor equiv}
	Let $\sG$ be an affine commutative group scheme over $\BD_{x'}$ endowed with a $\mu_r$-action. Then the category of $\sG$-torsors over $\BD_{x'}/\mu_r$ is equivalent to the category of torsors over the commutative group scheme $(\pi_{*}\sG)^{\mu_r}$ on $\BD_{x}$.
\end{lemma}
The proof in \cite{BS14} is actually given for an affine group scheme over a smooth curve endowed with a finite group action. The lemma follows from their argument.

\begin{theorem}\label{gen fiber}
	The generic fiber of $	h_{\sF}: \Hig_{\sG_{\sF}}\rightarrow H^{0}(C,\gc\times_{\mathbb{G}_{m}}\omega^{\times}_{C}(D))$ is gerbe banded by $\Pic(\sJ_{a}^{\sF})$. 
\end{theorem}
\begin{proof}
	Since we consider generic fibers, we may assume that $a(\BD_{x'})\subset\gc^{rs}$ for $x\in D$.
	By \cite[Corollary 17.6]{DG02}, we know that a Higgs bundle $(E|_{U},\theta|_{U})$ over $U=\ssC\ssm p^{-1}(D)$ is regular or not depends on the cameral cover $\tilde{U}$ (see \ref{def:cam cover}) which only depends $\chi(\theta|_{U})$. This implies that the generic fiber of Hitchin maps over $U$ is a gerbe banded by $\Pic\sJ_{a,U}$. Since over $\BD_{x'}$, we assume $a(\BD_{x'})\subset\gc^{rs}$(\footnote{As we can see, it is enough to assume the transversality of $a|_{\BD_{x'}}$} as in Proposition \ref{Diamond}), elements of generic fibers can be described as torsors over $\sJ_{a,\BD_{x'}}$ equipped with $\mu_r$ action. 
	
	In conclusion, we show that the generic fiber of the map:
	\begin{equation*}
	\Hig_{\mathscr{C}}\rightarrow \sH_{\sF}:=H^{0}(\mathscr{C},\omega^{\times}_{\mathscr{C}}\times_{\mathbb{G}_{m}}\gc)
	\end{equation*}
	is a gerbe banded by the Picard stack $\Pic\sJ_{a}$. In the proof of Proposition \ref{proposition:local equivalence of higgs}, to give the equivalence, we use conjugation of $\eta$. To be more precise:
	\begin{equation*}
	Ad_{\eta^{-1}}:\Aut^{\mathsf{C}(x)}(E_{x'})\rightarrow \hat{P}_{F(x)}
	\end{equation*} 
	is an isomorphism for each $x\in D$. Then we need to twist the group scheme $a(\BD_{x'})^{*}\sJ$ by $Ad_{\eta^{-1}}$ for $x\in D$. 
	The theorem then follows from Lemma \ref{equiva torsor equiv} and Lemma \ref{struc grp over curve}.    	
\end{proof}

As an application of the power of this orbifold approach, we show that in the case of $G=\SL_n$, our main result recovers in a much simpler manner one of the central results of  \cite{SWW19}.
\begin{corollary}\label{cor: app to SLn}
	If $G=\SL_{n}$, then the  generic geometric  fiber of a parabolic Hitchin map is the Picard variety of its normalized spectral curve.
\end{corollary}
\begin{proof}
	We sketch the proof:
	Let $a\in \sH_{\sF}$, the parabolic Hitchin base space. 
	In \cite{Ngo10}, it is shown that in our case the $\sJ_{a}$ over $\BD_{x'}$ is a finite type N\'eron model of its generic fiber. We need to show that $(\pi_{*}\sJ_{a,\BD_{x'}})^{\mu_r}$ is a finite type N\'eron model of it generic fiber. This is shown in Lemma \ref{fin Neron push forward} below.
	
	If $p:X_{a}\rightarrow X$ is a spectral curve, then the generic fiber of $(\pi_{*}\sJ_{a,\BD_{x'}})^{\mu_r}$ is $p_{*}\sO_{X_{a}}^{\times}$. If  
	$\tilde {p}:\tilde {X}_{a}\rightarrow X$ is the \emph{normalized} spectral curve, then it is known that the finite type N\'eron model is $\tilde{p}_{*}\sO_{\tilde X_{a}}$ for example see \cite[(3.6)]{CY01}. Thus we proved that $\Pic(\tilde{X}_{a})$ acting simply and transitively on the fiber of parabolic Hitchin maps over $a$.
\end{proof}

%
\begin{lemma}\label{fin Neron push forward}
	Let $\sO$, $\sK$, $\sO'$, $\sK'$, $\mu_r$ be as before. Let $T_{\sK'}$ be a torus over $\sK'$ , and $T$\ resp. $T^{ft}$ is its  N\'eron model over $\sO'$\ resp.  N\'eron model of finite \emph{type}. Then $(\pi_{*}T)^{\mu_r}$ \ resp. $(\pi_{*}T^{ft})^{\mu_r}$ is a  N\'eron model \ resp.  N\'eron model of finite \emph{type} of its generic fiber. 
\end{lemma}
\begin{proof}
	As $r$ is relatively prime to $p$, both $(\pi_{*}T)^{\mu_r}$ and $(\pi_{*}T^{ft})^{\mu_r}$ are smooth affine group scheme over $\sO$. 
	We first prove that $(\pi_{*}T)^{\mu_r}$ is a N\'eron model. As $\sO$ is strict Henselian, we only need to check $(\pi_{*}T)^{\mu_r}(\sO)=(\pi_{*}T)^{\mu_r}(\sK)$. The left hand side is $T(\sO')^{\mu_r}$. Since $T$ is a N\'eron model, $T(\sO')^{\mu_r}=T(\sK')^{\mu_r}$ which is the right hand side.
	
	We know that $(\pi_{*}T^{ft})^{\mu_r}(\sO)=(T^{ft}(\sO'))^{\mu_r}$ is the maximal bounded subgroup of $$((\pi_{*}T)(\sK))^{\mu_r}=(T(\sK'))^{\mu_r}$$
	Then by the characterization of N\'eron model of finite type in \cite[\S 3]{CY01}, we are done.
\end{proof}

\bibliographystyle{alpha}
\bibliography{ref}

\begin{thebibliography}{CMW17}

\bibitem[B{\'e}g80]{Be80}
Lucile B{\'e}gueri.
\newblock {\em Dualit{\'e} sur un corps local {\`a} corps r{\'e}siduel
  alg{\'e}briquement clos}, volume~4.
\newblock Gauthier-Villars, 1980.

\bibitem[Beh14]{Beh14}
K.~Behrend.
\newblock {\em Introduction to algebraic stacks}, page 1–131.
\newblock London Mathematical Society Lecture Note Series. Cambridge University
  Press, 2014.

\bibitem[BK18]{BK18}
David Baraglia and Masoud Kamgarpour.
\newblock On the image of the parabolic {Hitchin} map.
\newblock {\em Quarterly Journal of Mathematics}, 69(2):681--708, 2018.

\bibitem[BKV18]{BKV18}
David Baraglia, Masoud Kamgarpour, and Rohith Varma.
\newblock Complete integrability of the parahoric {Hitchin} system.
\newblock {\em International Mathematics Research Notices}, 2018.

\bibitem[BL94]{BL94}
Arnaud Beauville and Yves Laszlo.
\newblock Conformal blocks and generalized theta functions.
\newblock {\em Communications in mathematical physics}, 164(2):385--419, 1994.

\bibitem[BLR12]{BLR12}
Siegfried Bosch, Werner L{\"u}tkebohmert, and Michel Raynaud.
\newblock {\em N{\'e}ron models}, volume~21.
\newblock Springer Science \& Business Media, 2012.

\bibitem[BNR89]{BNR}
A.~Beauville, M.S. Narasimhan, and S.~Ramanan.
\newblock Spectral curves and the generalised theta divisor.
\newblock {\em J. Reine Angew. Math.}, 398:169--179, 1989.

\bibitem[Bou02]{Bour02}
N~Bourbaki.
\newblock Lie groups and lie algebras. chapters 4-6. translated from the 1968
  french original by a. pressley.
\newblock {\em Graduate Texts in Mathematics, Springer-Verlag}, 2002.

\bibitem[BR89]{BR89}
Usha Bhosle and A~Ramanathan.
\newblock Moduli of parabolic {G}-bundles on curves.
\newblock {\em Mathematische Zeitschrift}, 202(2):161--180, 1989.

\bibitem[BS14]{BS14}
V.~Balaji and C.~S. Seshadri.
\newblock Moduli of parahoric {$\mathcal{G}$}-torsors on a compact {Riemann}
  surface.
\newblock {\em Journal of Algebraic Geometry}, 24(1):1--49, 2014.

\bibitem[BT72]{BT72}
Fran{\c{c}}ois Bruhat and Jacques Tits.
\newblock Groupes r\'eductifs sur un corps local : I. donn\'ees radicielles
  valu\'ees.
\newblock {\em Publications Math\'ematiques de l'IH\'ES}, 41:5--251, 1972.

\bibitem[BT84]{BT84}
Fran{\c{c}}ois Bruhat and Jacques Tits.
\newblock Groupes r{\'e}ductifs sur un corps local: Ii. sch{\'e}mas en groupes.
  existence d'une donn{\'e}e radicielle valu{\'e}e.
\newblock {\em Publications Math{\'e}matiques de l'IH{\'E}S}, 60:5--184, 1984.

\bibitem[CGP12]{CGP12}
Vladimir Chernousov, Philippe Gille, and Arturo Pianzola.
\newblock Torsors over the punctured affine line.
\newblock {\em American Journal of Mathematics}, 134(6):1541--1583, 2012.

\bibitem[CMW17]{CSW17}
Sebastian Casalaina-Martin and Jonathan Wise.
\newblock An introduction to moduli stacks, with a view towards higgs bundles
  on algebraic curves.
\newblock {\em arXiv preprint arXiv:1708.08124}, 2:10--21, 2017.

\bibitem[CY01]{CY01}
Ching-Li Chai and Jiu-Kang Yu.
\newblock Congruences of n{\'e}ron models for tori and the artin conductor.
\newblock {\em Annals of mathematics}, pages 347--382, 2001.

\bibitem[DG02]{DG02}
R.~Donagi and D.~Gaitsgory.
\newblock The gerbe of higgs bundles.
\newblock {\em Transformation groups}, 7(2):109--153, 2002.

\bibitem[DS95]{DS95}
Vladimir~G. Drinfeld and Carlos~T. Simpson.
\newblock {$B$}-structures on {$G$}-bundles and local triviality.
\newblock {\em Mathematical Research Letters}, 2(6):823--829, 1995.

\bibitem[Edi92]{Ed92}
Bas Edixhoven.
\newblock N{\'e}ron models and tame ramification.
\newblock {\em Compositio Mathematica}, 81(3):291--306, 1992.

\bibitem[Gro16]{Groe16}
Michael Groechenig.
\newblock Moduli of flat connections in positive characteristic.
\newblock {\em Mathematical Research Letters}, 23(4):989--1047, 2016.

\bibitem[Har67]{Har67}
G{\"u}nter Harder.
\newblock Halbeinfache gruppenschemata {\"u}ber dedekindringen.
\newblock {\em Inventiones mathematicae}, 4(3):165--191, 1967.

\bibitem[Hei10]{Hein10}
Jochen Heinloth.
\newblock Uniformization of g-bundles.
\newblock {\em Mathematische Annalen}, 347(3):499--528, 2010.

\bibitem[Hit87]{Hit87S}
N.J. Hitchin.
\newblock Stable bundles and integrable systems.
\newblock {\em Duke Math. J.}, 54(1):91--114, 1987.

\bibitem[HS10]{HS10}
Jochen Heinloth and Alexander~HW Schmitt.
\newblock The cohomology rings of moduli stacks of principal bundles over
  curves.
\newblock {\em Documenta Mathematica}, 15:423--488, 2010.

\bibitem[Iwa66]{Iwa66}
Nagayoshi Iwahori.
\newblock Generalized tits system (bruhat decomposition) on p-adic semisimple
  groups.
\newblock 1966.

\bibitem[KNR97]{KNR97}
Shrawan Kumar, M.S. Narasimhan, and A.~Ramananthan.
\newblock Infinite grassmannians and moduli spaces of g-bundles.
\newblock In {\em Vector Bundles on Curves—New Directions}, pages 1--49.
  Springer, 1997.

\bibitem[Kos63]{Kos63}
Bertram Kostant.
\newblock Lie group representations on polynomial rings.
\newblock {\em American Journal of Mathematics}, 85(3):327--404, 1963.

\bibitem[MT11]{MT11}
Gunter Malle and Donna Testerman.
\newblock {\em Linear algebraic groups and finite groups of Lie type}, volume
  133.
\newblock Cambridge University Press, 2011.

\bibitem[Ngo10]{Ngo10}
Bao~Chau Ngo.
\newblock Le lemme fondamental pour les alg\`ebres de {Lie}.
\newblock {\em Publications Math\'ematiques de l'IH\'ES}, 111(1):1--169, 2010.

\bibitem[Ric17]{Ric17}
Simon Riche.
\newblock Kostant section, universal centralizer, and a modular derived satake
  equivalence.
\newblock {\em Mathematische Zeitschrift}, 286(1-2):223--261, 2017.

\bibitem[SWW19]{SWW19}
Xiaoyu Su, Bin Wang, and Xueqing Wen.
\newblock Parabolic hitchin maps and their generic fibers.
\newblock {\em arXiv preprint arXiv:1906.04475}, 2019.

\bibitem[Tit79]{Tits79}
Jacques Tits.
\newblock Reductive groups over local fields.
\newblock In {\em Automorphic forms, representations and L-functions (Proc.
  Sympos. Pure Math., Oregon State Univ., Corvallis, Ore., 1977), Part},
  volume~1, pages 29--69, 1979.

\bibitem[Vel72]{Vel72}
Ferdinand~D Veldkamp.
\newblock The center of the universal enveloping algebra of a lie algebra in
  characteristic $ p$.
\newblock In {\em Annales scientifiques de l'{\'E}cole Normale Sup{\'e}rieure},
  volume~5, pages 217--240, 1972.

\end{thebibliography}
\end{document}